\documentclass[pdflatex,sn-basic]{sn-jnl}

\usepackage{graphicx}
\usepackage{amsmath}
\usepackage{cleveref}
\usepackage[inline]{enumitem}
\usepackage{subfig}
\usepackage[super]{nth}

\newtheorem{theorem}{Theorem}
\newtheorem{proposition}{Proposition}
\newtheorem{corollary}{Corollary}
\newtheorem{lemma}{Lemma}

\newtheorem{Def}{Definition}
\newcommand{\Cov}{\operatorname{Cov}}
\newcommand{\Var}{\operatorname{Var}}
\newcommand{\E}{\operatorname{E}}

\newcommand{\Tr}{\operatorname{Tr}}
\newcommand{\Xbar}[1]{\overline{{#1}}}
\newcommand{\myspace}{\hspace{4pt}}

\normalbaroutside
\begin{document}

\title[On confidence intervals for precision matrices]{On confidence intervals for precision matrices and the eigendecomposition of covariance matrices}


\author*[1]{\fnm{Teodora} \sur{Popordanoska}}\email{teodora.popordanoska@kuleuven.be}

\author[2]{\fnm{Aleksei} \sur{Tiulpin}}

\author[3]{\fnm{Wacha} \sur{Bounliphone}} 

\author[1]{\fnm{Matthew} \sur{B. Blaschko}}

\affil[1]{ESAT-PSI, KU Leuven, Belgium}
\affil[2]{Research Unit of Medical Imaging, Physics and Technology, University of Oulu, Finland}
\affil[3]{IPSOS, Social Intelligence Analytics, France}

\abstract{The eigendecomposition of a matrix is the central procedure in probabilistic models based on matrix factorization, for instance principal component analysis and topic models. Quantifying the uncertainty of such a decomposition based on a finite sample estimate is essential to reasoning under uncertainty when employing such models. 
This paper tackles the challenge of computing confidence bounds on the individual entries of eigenvectors of a covariance matrix of fixed dimension. 
Moreover, we derive a method to bound the entries of the inverse covariance matrix, the so-called precision matrix.
The assumptions behind our method are minimal and require that the covariance matrix exists, and its empirical estimator converges to the true covariance.
We make use of the theory of U-statistics to bound the $L_2$ perturbation of the empirical covariance matrix.
From this result, we obtain bounds on the eigenvectors using Weyl's theorem and the eigenvalue-eigenvector identity and we derive confidence intervals on the entries of the precision matrix using matrix inversion perturbation bounds. 
As an application of these results, we demonstrate a new statistical test, which allows us to test for non-zero values of the precision matrix. We compare this test to 
the well-known Fisher-z test for partial correlations, and demonstrate the soundness and scalability of the proposed statistical test, as well as its application to real-world data from medical and physics domains. }

\keywords{confidence intervals, eigendecomposition, precision matrix, distribution-independent, eigenvalue-eigenvector identity}

\maketitle

\section{Introduction}
Estimating confidence intervals on the eigendecomposition of an empirical covariance matrix $\hat \Sigma$ is significant in various applications, including dimensionality reduction techniques, canonical correlation analysis and topic modeling.
Given a finite sample $\{x_i\}_{i=1}^n$ drawn i.i.d.\ from some unknown distribution with $x_i\in \mathbb{R}^p$, the task is to find confidence intervals on the individual entries of the eigenvectors $u_j$ such that $\Sigma u_j = \lambda_j u_j $, $\Sigma$ being the \emph{population} covariance matrix to which the empirical estimate $\hat{\Sigma}$ converges as $n$ goes to infinity.
That such confidence intervals are possible is due to the recently codified eigenvalue-eigenvector identity \citep{denton2019eigenvectors}, a fundamental relationship between eigenvalues and individual entries of eigenvectors.

A closely related problem is prominent in the field of graphical model structure discovery, where there is a general interest in estimating $\hat \Sigma^{-1} = \hat \Theta$, a precision matrix (PM).
In the case of multivariate Gaussian distributions, the PM encodes the conditional dependence structure of the graphical model, i.e., non-zero entries in the precision matrix can be shown to correspond to the edges in the graph \citep{dempster1972covariance}. When the distribution is not Gaussian, such a direct link between $\Theta$ and the graphical model structure cannot be drawn, but still allows for reasoning about linear relationships among the variables. We note that in learning the structure of the graphical model, one might also want to assess the significance of the discovered relationships. Specifically, one can assess whether partial correlations, directly computable from $\Theta$, are significant.

Hypothesis testing with statistical measures of dependence is a relatively well developed field with a number of general results. Classical tests such as Spearman's $\rho$~\citep{spearman1904proof}, Kendall's $\tau$~\citep{kendall1938new}, R\'{e}nyi's $\alpha$~\citep{renyi1961measures} and Tsallis' $\alpha$~\citep{tsallis1988possible} are widely applied. Recently, for multivariate non-linear dependencies, novel statistical tests were introduced and some prominent examples include the kernel mutual information \citep{gretton2003kernel}, the generalized variance and kernel canonical correlation analysis~\citep{bach2003kernel}, the Hilbert-Schmidt independence criterion~\citep{GreBouSmoSch05}, the distance based correlation~\citep{szekely2007measuring} and rankings~\citep{heller2012consistent}. However, testing for \emph{conditional} dependence is challenging, and only few dependence measures have been generalized to the conditional case~\citep{fukumizu2007kernel,FukumizuBachJordan2009,zhang2012kernel}. The existing techniques rely on either having assumptions, such as Gaussianity of the data or sparsity of the dependence structure in the setting that the number of variables being larger than the number of samples, $p \gg n$. An alternative and robust solution in this case is the bootstrap~\citep{efron1992bootstrap}, but it is computationally challenging for large $n$.

The paper is structured as follows. First, we review existing methods both for estimating eigenvector bounds and precision matrix bounds. In the subsequent section, using the theory of U-statistics~\citep{hoeffding1948class}, we derive an estimator for the covariance of the empirical covariance estimator. In the following two sections we leverage this result to:
\begin{enumerate}[nosep]
    \item Develop a novel, statistically and computationally efficient framework for estimating the bounds on the eigenvalues and eigenvectors using the eigenvalue-eigenvector identity~\citep{denton2019eigenvectors}.
    \item Derive confidence intervals for the precision matrix using matrix inversion perturbation bounds. From these bounds, we develop a procedure for hypothesis testing of whether an entry of the precision matrix is non-zero based on a data sample from the joint distribution $P_X$. Compared to the prior art, we focus on designing a test for the case when $n \gg p$, and in particular ensure that the test has computational complexity \emph{linear} in $n$, while making two \emph{minimal} distributional assumptions
    \begin{enumerate*}[label=(\roman*)]
        \item the covariance matrix $\Sigma$ exists
        \item an unbiased estimator $\hat \Sigma$ converges to $\Sigma$.
    \end{enumerate*}
    The proposed test is asymptotically consistent without the need to set a computationally expensive bootstrap procedure or any hyperparameters.
\end{enumerate}

\section{Related work}
Three lines of work related to ours exist. Firstly, methods for estimating eigenvector bounds in the covariance matrix. Secondly, methods for computing bounds on the entries of the precision matrix.  Finally, there exist statistical tests for partial correlations. We review the relevant studies below and compare them to our work.

Multiple studies focused on studying the limiting behavior of eigenvalues and eigenvectors of covariance matrices. As such, \cite{xia2013convergence} established a convergence rate of the eigenvector empirical spectral distribution, which was recently improved by~\cite{Xi_2020}, assuming a finite \nth{8} moment condition of the underlying data distribution. Prior to that line of work,~\cite{Ledoit_2009} quantified the deviation of the sample eigenvectors from their population values by generalizing the Mar\v{c}enko-Pastur equation~\cite{Marcenko_1967} and assuming \nth{12} finite moment and extra constraints. Additional results with similar assumptions can be found in~\cite{Bai_2007,Silverstein_2000,Silverstein_1989,Silverstein_1984}.

With relation to estimating the bounds on the precision matrix, in the case of  multivariate Gaussian distribution, there exist a range of structure discovery techniques in high-dimensional settings, where the number of samples is less than the dimensionality of the data.  Estimation of such high-dimensional models has been the focus on recent research~\citep{schafer2005shrinkage,li2006gradient,meinshausen2006,banerjee2008model,Friedman01072008,ravikumar2011high} where methods impose a sparsity constraint on the entries of the inverse covariance matrix. The graphical lasso (GL)~\citep{Friedman01072008} uses $l_1$-based regularization to provide estimation of the precision matrix under some structural assumptions. The consequence of this attractive method has been the development of diverse statistical hypothesis tests \citep{g2013adaptive,lockhart2014significance,jankova2015confidence}. However, they are designed and optimized for high-dimensional settings, while we focus instead on the case where $n \gg p$.

For a low-dimensional setting,~\cite{Williams_2019} report superior performance compared to GL by using Fisher's Z-transformation to obtain confidence intervals for detecting non-zero relationships from the precision matrix.
Nonetheless, each of these methods explicitly assumes that the data distribution is multivariate Gaussian. On the other hand, our method does not impose any additional constraints on the distribution.

Testing for significant non-zero partial correlations can be done without any assumptions on the data distribution using a permutation test, which is itself a highly computationally expensive procedure. However, if one is also interested in obtaining error bounds, the only available distribution-independent method until now, to the best of our knowledge, is a bootstrap technique~\citep{efron1992bootstrap}, which becomes computationally expensive in practice when $n$ is large. 


\section{Bounding the perturbation of the covariance matrix }
\subsection{A U-statistic estimator of the cross-covariance}\label{EJS:subsec:ustatistic}
In this section we make use of the classic theory of U-statistics estimators, which allows a minimum-variance unbiased estimation of a parameter.
Most of the materials can be found in \cite{hoeffding1948class}, \cite[Chap. 5]{Serfling1981}, \cite[Chap. 6]{lehmann1999elements} and \cite{lee1990u}. 
Suppose we have a sample $\operatorname{X_m} = \{X_1, ..., X_m\}$ of size $m$ and dimension $p$ drawn i.i.d.\ from a distribution $P_X$.  
A $U$-statistics concerns an unbiased estimator of a parameter $\theta$ of $P_X$ using $\operatorname{X_m}$.
That is, $\theta$ may be represented as 
\begin{equation}
    \theta=\E[h(X_1, ..., X_m)]
\end{equation}
for some function $h$, called a kernel of order $m$.

\begin{Def}(U-statistic, \cite[Chap. 5]{Serfling1981})
Given a kernel $h$ of order $m$ and a sample $ \operatorname{X_n} = \{X_1, ..., X_n\}$ of size $n \geq m$, the corresponding $U$-statistic for estimation of $\theta$ is obtained by averaging the kernel $h$ symetrically over the observations:
\begin{equation}
\hat{U} := \dfrac{1}{(n)_m} \sum_{i^n_m} h(X_{i_1}, ..., X_{i_m})
\end{equation}
where the summation ranges over the set $i^n_m$ of all $\dfrac{n!}{(n-m)!}$ permutations of size $m$ chosen from $(1, ..., n)$ and $(n)_m$ is the Pochhammer symbol $(n)_m := \dfrac{n!}{(n-m)!}$.
\end{Def}
\begin{Def}($U$-statistic estimator of the covariance) \label{EJS:def:estimator_Ustat_of_covariance}  Let $u_r = (X_{i_r},X_{j_r})^T$ be ordered pairs of samples, with $ 1 \leq r \leq n $ and $1 \leq i, j \leq p$.  Consider $\Sigma = \Cov(X_{i},X_{j})$, the covariance functional between $X_{i}$ and $X_{j}$ and $h$, the kernel of order 2 for the functional $\Sigma$ such that
\begin{equation}
h(u_1,u_2) = \dfrac{1}{2}(X_{i_1} - X_{i_2})(X_{j_1} - X_{j_2}).
\label{EJS:eq:_kernel_Ustat_of_covariance}
\end{equation}
The corresponding $U$-statistic estimator of the covariance $\Sigma$ is
\begin{align}
\hat{\Sigma} &= 
\dfrac{1}{n-1} \sum_{r=1}^n (X_{i_r} - \bar{X_i})(X_{j_r} - \bar{X_j}) ,
\label{EJS:eq:estimator_Ustat_of_covariance}
\end{align}
where $ \bar{X_i} = \frac{1}{n}  \sum_{s=1}^n X_{is}$. $\hat{\Sigma}$ can be computed in linear time.
\end{Def}

\subsection{U-statistic based convergence of the empirical covariance matrix}

We focus here on $U$-statistic estimates of $\hat{\Sigma}$ and its asymptotic normal distribution.\footnote{For simplicity, we chose to use U-statistics and the asymptotic normality. An alternative approach would be to use a non-asymptotic result, e.g. based on Berry-Esseen \citep{feller}, however, as can be seen in Section \ref{sec:bounding_eigvectors}, we use the perturbation bounds in combination with interval arithmetic and therefore any difference in the two approaches will not result in a meaningful change in the final bounds.} We develop the full covariance between the elements of $\hat{\Sigma}$, which we denote $\Cov(\hat{\Sigma}) \in \mathbb{R}^{\frac{p(p+1)}{2} \times \frac{p(p+1)}{2}}$, where the size is due to the symmetry of $\hat{\Sigma}$. We denote $U(A)$ the function returning the upper triangular part and diagonal of a matrix $A$.

\begin{theorem}\label{aistats2021:asymptotic_normal_dist_for_Ustatistics_Cov}(Joint asymptotic normality distribution of the covariance matrix, \citep{hoeffding1948class}) 
For all $(i,j,k,l)$ range over each of the $p$ variates in a covariance matrix $\hat{\Sigma}$, if $\Var(\hat{\Sigma}_{ij}) > 0$ and $\Var(\hat{\Sigma}_{kl}) > 0$, then $\left[ \hat{\Sigma}_{ij} ; \hat{\Sigma}_{kl} \right]^T$ converges in distribution (as $n \rightarrow \infty$) to a Gaussian random variable 
\begin{align}
& n^{\frac{1}{2}}
\begin{pmatrix}
\hat{\Sigma}_{ij} - \Sigma_{ij}  \\
\hat{\Sigma}_{kl} - \Sigma_{kl}  \\
\end{pmatrix}
\overset{d}{\longrightarrow} \mathcal{N} 
\begin{pmatrix}
\begin{pmatrix}
0  \\
0  \\
\end{pmatrix}, 
K,
\end{pmatrix},
\end{align}
where 
\begin{align}
K=\begin{pmatrix} 
\Var(\hat{\Sigma}_{ij}) & \Cov(\hat{\Sigma}_{ij},\hat{\Sigma}_{kl}) \\ 
\Cov(\hat{\Sigma}_{ij},\hat{\Sigma}_{kl}) & \Var(\hat{\Sigma}_{kl})
\end{pmatrix} 
\end{align}
\end{theorem}

We note respectively $h$ and $g$ the corresponding kernels of order 2 for the two unbiased estimates $\hat{\Sigma}_{ij}$ and  $\hat{\Sigma}_{kl}$, where $h(u_1,u_2) = \dfrac{1}{2} \left( X_{i_1} - X_{i_2} \right)  \left( X_{j_1} - X_{j_2} \right)$
with $u_r=(X_{i_r}$, $X_{j_r})^T$ and $g(v_1,v_2) = \dfrac{1}{2} \left( X_{k_1} - X_{k_2} \right)  \left( X_{l_1} - X_{l_2} \right)$ with $v_r=(X_{k_r},X_{l_r})^T$. Then, we state the following theorem.

\begin{theorem}\label{aistats2021:theorem:covariance_covariance_ustatistic}
(Covariance of the $U$-statistic for the covariance matrix)
The low variance, unbiased estimates of the covariance between two $U$-statistics estimates $\hat{\Sigma}_{ij}$ and $\hat{\Sigma}_{kl}$, where $(i,j,k,l)$ range over each of the $p$ variates in a covariance matrix of $\hat{\Sigma}$ $\Cov(\hat{\Sigma}):=\Cov(\hat{\Sigma}_{ij},\hat{\Sigma}_{kl})$ is

\begin{align}
\Cov(\hat{\Sigma}_{ij},\hat{\Sigma}_{kl})=\binom{n}{2}^{-1} \left( 2 (n-2) \zeta_1  \right) + \mathcal{O}(n^{-2}),
\label{aistats2021:eq:Covariance_of_Ustat_estimator}
\end{align}
where $\zeta_1 = \Cov \left(  \E_{u_2}[h(u_1,u_2)], \E_{v_2}[g(v_1,v_2)] \right)$. 
There are seven exhaustive cases, derived in
Appendix~\ref{aistats2021:appendix:subsec:derivation_sevencases},
which are used to estimate Eq.~\eqref{aistats2021:eq:Covariance_of_Ustat_estimator} for all $1 \leq i,j,k,l \leq p$ through simple variable substitution.  Each of these cases has computation linear in $n$. 
\end{theorem}

\begin{lemma}\label{aistats2021:lemma:norm2_less_normfro}
With probability at least $1-\delta$, we (asymptotically) have the two following inequalities
\begin{align}
\Vert \Sigma - \hat{\Sigma} \Vert_2 &\leq \sqrt{2 \lambda_{\max}} \Phi^{-1} \left( 1 - \delta/2 \right)  \label{aistats2021:eq:eigenvalueBound}\\
& \leq \sqrt{2 \Tr[\Cov(\hat{\Sigma})]} \Phi^{-1} \left( 1 - \delta/2 \right) \label{aistats2021:eq:TraceBound}
\end{align}
where $\Phi(\cdot)$ is the CDF of $\mathcal{N}(0,1)$ and $\lambda_{\max}$ is the largest eigenvalue of $\Cov(\hat{\Sigma})$. 
\end{lemma}
\begin{proof}
As $\hat{\Sigma}$ is a $U$-statistic, we have that $U(\hat{\Sigma})$, a vector containing its upper diagonal component (including the diagonal), is Gaussian distributed with covariance $\Cov(\hat{\Sigma})$ (cf.~Thm~\ref{aistats2021:asymptotic_normal_dist_for_Ustatistics_Cov},~\ref{aistats2021:theorem:covariance_covariance_ustatistic}).  Therefore, with probability at least $1-\delta$, \begin{equation}\label{NIPS2016:eq:largesteigenvalueBound}
\| U(\Sigma) - U(\hat{\Sigma}) \|_2 \leq \sqrt{\lambda_{\text{max}}} \Phi^{-1}\left(1 - \delta/2 \right)
\end{equation}
and furthermore
\begin{equation}
\| \Sigma - \hat{\Sigma} \|_F \leq \sqrt{2} \| U(\Sigma) - U(\hat{\Sigma}) \|_2
\end{equation}
which combined with the fact that $\| \cdot \|_2 \leq \| \cdot \|_F$ yields the desired result. 
\end{proof}

\section{Bounding the eigenvectors of the covariance matrix}
\label{sec:bounding_eigvectors}
\begin{theorem}[Eigenvalue-eigenvector identity \citep{denton2019eigenvectors}]
    \label{aistats2021:eigenvalue_eigenvector_identity}
    Let $\Sigma$ be a $p\times p$ Hermitian matrix.
    Denote by $V_{ij}$ the $j$th component of the $i$th eigenvector of $\Sigma$, and $\lambda_i(\Sigma)$ be the $i$th eigenvalue.  The following identity holds:
\begin{equation}
    |V_{ij}|^2 \prod_{k=1; k\neq i}^p\left(\lambda_i(\Sigma)-\lambda_k(\Sigma)\right) = \prod_{k=1}^{p-1} \left( \lambda_i(\Sigma) - \lambda_{k}(M_j) \right), \label{eq:EigenvalueEigenvectorIdentity}
\end{equation}

where $M_j$ is the $p-1\times p-1$ minor formed from $\Sigma$ by deleting its $j$th row and column.
\end{theorem}
Assume that we have a perturbed observation $\hat{\Sigma}$ satisfying $\|\hat{\Sigma} - \Sigma \|_2 \leq \varepsilon$ with high probability for some known $\varepsilon$.  We may construct such a $\varepsilon$, e.g.\ when estimating a covariance matrix from a finite sample.
\begin{theorem}[Weyl's Theorem \citep{Weyl1912}]\label{aistats2021:weyls_theorem}
For two positive definite matrices $\hat{\Sigma}$ and $\Sigma$, if
\begin{equation}
    |\lambda_k(\hat{\Sigma}) - \lambda_k(\Sigma)| \leq \|\hat{\Sigma}-\Sigma\|_2\leq \varepsilon
\end{equation}

where $0<\varepsilon<\lambda_k(\Sigma)\ \forall k$, then
\begin{align}
\lambda_k(\hat{\Sigma}) - \varepsilon \leq \lambda_k(\Sigma) \leq \lambda_k(\hat{\Sigma}) + \varepsilon \ \forall k.
\end{align}\label{aistats2021:eigenvalues_bounds_sigma}
\end{theorem}
\begin{corollary}
\label{aistats2021:corollary:eigenvalues_bounds_pm}
If $\|\hat{\Sigma}-\Sigma\|_2\leq \varepsilon$,
\begin{align}
    (\lambda_k(\hat{\Sigma}) + \varepsilon)^{-1} \leq \lambda_k(\Sigma)^{-1}.
\end{align}
Furthermore, if $\lambda_{k}(\hat{\Sigma}) > \varepsilon$,
    \begin{align}
        \lambda_k(\Sigma)^{-1} \leq (\lambda_k(\hat{\Sigma}) - \varepsilon)^{-1} .
    \end{align}
\end{corollary}

\begin{proposition}\label{aistats2021:eigenvector_bounds}
Assuming that $\|\hat{\Sigma}-\Sigma\|_2\leq \varepsilon$ and the fact that the eigenvectors are orthonormal, it follows simultanously for all $(i, j)$ that
\begin{equation}
   |V_{ij}|^2  
   \leq
   \begin{cases}
    1 & \text{ if } \exists k\neq i: |\lambda_i(\hat{\Sigma})-\lambda_k(\hat{\Sigma)}| \leq 2\varepsilon \\
   \min (\alpha_{ij}, 1) & \text{otherwise},
   \end{cases}
\end{equation}
and 
\begin{equation}
    |V_{ij}|^2  \geq \begin{cases}
    0 & \text{ if } \exists k : | \lambda_i(\hat\Sigma) - \lambda_{k}(\hat{M}_j) | \leq 2 \varepsilon \\
    \beta_{ij} & \text{ otherwise},
    \end{cases} \label{eq:MainEigenvectorLB}
\end{equation}
where
\begin{equation}
    \alpha_{ij} = \frac{\prod_{k=1}^{p-1} \left( |\lambda_i(\hat{\Sigma}) - \lambda_{k}(\hat{M}_j)| + 2\varepsilon \right)}{\prod_{k=1; k\neq i}^p\left(|\lambda_i(\hat{\Sigma})-\lambda_k(\hat{\Sigma})| - 2\varepsilon \right)} \\
    \beta_{ij} = \frac{\prod_{k=1}^{p-1} \left( |\lambda_i(\hat\Sigma) - \lambda_{k}(\hat{M}_j)| - 2 \varepsilon  \right)}{\prod_{k=1; k\neq i}^p\left(|\lambda_i(\Sigma)-\lambda_k(\Sigma)| + 2\varepsilon\right) }
\end{equation}
\end{proposition}
\begin{proof}
We first note that $\|\hat{M}_j - M_j\|_2 \leq \|\hat{\Sigma} - \Sigma\|_2$, where $\hat{M}_j$ denotes the corresponding minor of $\hat{\Sigma}$, which follows directly from the subadditivity of the norm.  This indicates we can also use $\varepsilon$ to bound the perturbation on the eigenvalues of $\hat{M}_j$, though it would be possible to compute a slightly tighter bound on $|V_{ij}|^2$ at the expense of extra computation for a tighter bound on $\|\hat{M}_j - M_j\|_2$.

We subsequently note that the quantity to be bounded is non-negative, which allows us to replace all eigenvalue differences with the absolute value.
The rest of the inequality follows by a simple application of interval arithmetic \citep{IntervalAnalysis2009} based on eigenvalue bounds from Weyl's theorem.
\end{proof}
\begin{corollary}
If the lower bound on $|V_{ij}|^2$ is greater than zero, we have that the sign of $V_{ij}$ is equal to the sign of $\hat{V}_{ij}$ and we can recover an upper and lower bound on $V_{ij}$ directly.  Otherwise, we conclude that the signed lower bound is the negative of the unsigned upper bound.
\label{cor:signedBounds}
\end{corollary}

\subsection{Orthonormality constraints}\label{aistats2021:sec:orthonormality}

The bounds on the eigenvectors described above are valid, but do not exploit the orthonormality property of the eigenvectors. Here we show how one can make use of it, and tighten the bounds. 

Let $\hat{\beta}$ (respectively $\hat{\alpha}$) denote a previously obtained lower bound (respectively upper bound) on $V \odot V$, where $\odot$ denotes the Hadamard product.  From the fact that the norm of each eigenvector is equal to one, we additionally obtain $0 \leq |V_{ij}|\leq 1$, and
\begin{equation} \label{eq:TightenedBoundsFromNormConstraints}
1 - \sum_{k\neq i}\hat{\alpha}_{kj}  \leq |V_{ij}|^2\leq 1 - \sum_{k\neq i}\hat{\beta}_{kj} .
\end{equation}

Orthogonality of the eigenvectors,
\begin{align}
    \sum_k V_{ki} V_{kj}=0 ,
\end{align}
implies additional constraints on $V_{ij}$.  Let $\alpha \geq V \geq \beta$ be previously obtained (probabilistic) signed bounds on $V$ (here the inequalities are taken to be element wise).
Using interval arithmetic \citep{IntervalAnalysis2009} we can compute bounds
\begin{align}
    \mu_{ijl} \leq \sum_{k\neq l} V_{ki} V_{kj} \leq \nu_{ijl} ,
\end{align}
and subsequently infer
\begin{align}\label{eq:TightenedBoundsFromOrthogonalityConstraints}
    \min(V_{li}\alpha_{lj}, V_{li}\beta_{lj}) + \mu_{ijl} \leq 0 \leq \max(V_{li}\alpha_{lj}, V_{li}\beta_{lj}) + \nu_{ijl} ,
\end{align}
which in turn leads to additional upper and lower bounds on $V_{li}$, the forms of which depend on the signs of the coefficients.  We enumerate the cases here:
\begin{enumerate}
    \item $\alpha_{lj} < 0$:
    We start with the constraint
    \begin{align}   \min(V_{li}\alpha_{lj}, V_{li}\beta_{lj})  \leq -\mu_{ijl} .
    \end{align}
    Assuming  $V_{li}\geq 0$ yields
    \begin{align}   V_{li}\geq -\frac{\mu_{ijl}}{\beta_{lj} }
    \end{align}
    Assuming $V_{li}<0$ yields
    \begin{align} V_{li}  \geq -\frac{\mu_{ijl}}{\alpha_{lj}}
    \end{align}
    If we have already constrained the sign of $V_{li}$ we can now directly use one of these two inequalities, but we know in any case that
    \begin{align}
   V_{li}  \geq \min\left(-\frac{\mu_{ijl}}{\alpha_{lj} }, -\frac{\mu_{ijl}}{\beta_{lj}} \right) .
    \end{align}
    
    A similar line of reasoning yields 
    \begin{align}
     V_{li}  \leq \max\left(-\frac{\nu_{ijl}}{\alpha_{lj}}  , -\frac{\nu_{ijl}}{\beta_{lj}} \right) ,
    \end{align}
    which can be similarly sharpened if the sign of $V_{li}$ is known.
    \item $\beta_{lj}>0$:
    Analogous to the previous case, we obtain
    \begin{align}
        \min\left( -\frac{\nu_{ijl}}{\alpha_{lj}} ,  -\frac{\nu_{ijl}}{ \beta_{lj}  }  \right)\leq V_{li} \leq \max\left( -\frac{\mu_{ijl}}{\alpha_{lj}} , -\frac{\mu_{ijl}}{\beta_{lj}} \right) .
    \end{align}
    \item $\alpha_{lj} > 0 \wedge \beta_{lj}< 0$:
    Considering the first inequality
    \begin{align}
        \min(V_{li}\alpha_{lj}, V_{li}\beta_{lj})  \leq - \mu_{ijl}.
    \end{align}
    Assume $V_{li}\geq 0$:
    \begin{align}
        V_{li}  \geq - \frac{\mu_{ijl}}{ \beta_{lj} }
    \end{align}
    Assume $V_{li}<0$:
    \begin{align}
        V_{li}  \leq - \frac{\mu_{ijl}}{\alpha_{lj} }
    \end{align}
    Considering the second inequality
    \begin{align}
    \max(V_{li}\alpha_{lj}, V_{li}\beta_{lj})  \geq -\nu_{ijl} .
    \end{align}
    Assuming $V_{li}\geq 0$ yields
    \begin{align}
    V_{li} \geq -\frac{\nu_{ijl}}{ \alpha_{lj} } .
    \end{align}
    Assuming $V_{li}<0$ yields
    \begin{align}
     V_{li}  \leq -\frac{\nu_{ijl}}{ \beta_{lj} } ,
    \end{align}
    and we can only add an additional constraint if the sign of $V_{li}$ is known.
    \item $\beta_{lj}=0$ (the case of $\alpha_{lj}=0$ is symmetric):
    The first constraint simplifies to
    \begin{align}
    \min(V_{li}\alpha_{lj}, 0) \leq - \mu_{ijl} .
    \end{align}
    If $\mu_{ijl}$ is negative, this is satisfied trivially, so under the assumption that $\mu_{ijl}>0$, this reduces to 
    \begin{align}
    V_{li} \leq - \frac{\mu_{ijl}}{ \alpha_{lj} } .
    \end{align}

    The second constraint simplifies to
    \begin{align}
    \max(V_{li}\alpha_{lj}, 0) \geq - \nu_{ijl} ,
    \end{align}
    and implies (under the assumption that $ \nu_{ijl}<0 $)
    \begin{align}
    V_{li} \geq - \frac{\nu_{ijl}}{\alpha_{lj}} .
    \end{align}
\end{enumerate}
We note that Equations~\eqref{eq:TightenedBoundsFromNormConstraints} and~\eqref{eq:TightenedBoundsFromOrthogonalityConstraints} may compliment each other, so we may iterate their application until there is no more improvement in the confidence intervals, or until improvement falls below a given tolerance.
The orthonormality constraints are summarized in Appendix~\ref{section:orthonormality}.

\subsection{Bounding the precision matrix via the bounds on the eigendecomposition}
From the results described in the previous sections, we can derive bounds on the precision matrix $\Theta := \Sigma^{-1}$. To bound 
\begin{align}
    \Theta_{ij}  
    = \sum_{k=1}^p V_{ik} \lambda_{k}(\Sigma)^{-1} V_{jk}  , \label{eq:PrecisionPeturbation}
\end{align}
we have upper and lower bounds for each of the $\lambda_{k}(\Sigma)$ (from Weyl's theorem) and $V$ (from the eigenvalue-eigenvector identity), so we may simply employ interval arithmetic \citep{IntervalAnalysis2009} on the operations in Equation~\eqref{eq:PrecisionPeturbation} to compute upper and lower bounds on $\Theta_{ij}$.  If the upper and lower bounds have the same sign, we know that the partial correlation is significantly non-zero.

We note, however, that for the purpose of computing bounds on the precision matrix, a more direct approach is to use the result bounding the $L_2$ norm of the covariance matrix together with perturbation bounds for a matrix inverse. This method yields tighter bounds and will be described in the following section.

\section{Bounding the precision matrix using general perturbation bounds}
\label{section:pm_pertubration_bounds}

To derive a method to bound the entries of the precision matrix, we make use of general perturbation bounds from \cite{Kereta21,Xu2020,Li2018}. 
We will use the fact that for matrices $A$ and $B$:
\begin{align}
    \|A\|_2=\sigma_{max}(A) = \sqrt{\lambda_{max}(A^H A)},
\end{align}
where $\sigma$ is the singular value of $A$ and  $A^H$ is the conjugate transpose of $A$. Furthermore, for symmetric PSD matrices (such as covariances) it holds that:
\begin{align}
    \sigma_{max}(A) = \lambda_{max}(A)
\end{align}

\paragraph{General perturbation bound}
We make use of the inequality at the top of the page 4 of \cite{Kereta21}, which allows us to get the bound on the difference between the empirical and the true precision matrices. 
In \cite{Kereta21} they use Moore-Penrose pseudo-inverse notation, but since we deal with covariances, we can simply consider inverse matrices. Let $\Delta$ denote a random perturbation and $\hat \Sigma = \Sigma + \Delta$. Considering one of the inequalities in \cite{Kereta21}, we have that:
\begin{align}
    \|\hat \Sigma^{-1} - \Sigma^{-1} \|_F \leq \|\hat \Sigma^{-1} \|_2\|\Sigma^{-1}\Delta\|_F.
    \label{eq:general_perturbation_bound}
\end{align}

Using the strong sub-multiplicativity property of the Frobenius norm \citep[Chap. 2.5]{drineas2017lectures}), namely that $\|AB\|_F \leq \|A\|_2\|B\|_F$, we can rewrite the r.h.s. of Equation~\ref{eq:general_perturbation_bound} as follows:
\begin{align}
    \|\hat \Sigma^{-1} \|_2\|\Sigma^{-1}\Delta\|_F \leq \|\hat \Sigma^{-1} \|_2\|\Sigma^{-1}\|_2\|\Delta\|_F,
\end{align}


It is clear that $\|\hat \Sigma^{-1} \|_2$ is computable, 
and 
\begin{align}
    \|\hat \Sigma^{-1} \|_2 = \lambda_{max}(\hat \Sigma^{-1}) = \lambda_{min}(\hat \Sigma)^{-1}.
\end{align}

We will now make use of  Corollary~\ref{aistats2021:corollary:eigenvalues_bounds_pm} to derive the bound for $\|\Sigma^{-1}\|_2$. We know that 
\begin{align}
    \|\Sigma^{-1}\|_2 = \lambda_{max}( G^{-1})
    = \lambda_{min}(\Sigma)^{-1}
\end{align}

Let's assume that $\lambda_{min}(\hat{\Sigma}) > \sqrt{2 \lambda_{\max}} \Phi^{-1} \left( 1 - \delta/2 \right)$. Then from Corollary~\ref{aistats2021:corollary:eigenvalues_bounds_pm}
\begin{align}
    \|\Sigma^{-1}\|_2 \leq \frac{1}{\lambda_{min}(\hat \Sigma) - \sqrt{2 \lambda_{\max}} \Phi^{-1} \left( 1 - \delta/2 \right)}.
\end{align}

Making use of Lemma~\ref{aistats2021:lemma:norm2_less_normfro} to get the bound for $\|\Delta\|_F = \| \hat{\Sigma} - \Sigma \|_F$ allows us to finally compute the bound:
\begin{align}
    \|\hat \Sigma^{-1} \|_2\|\Sigma^{-1}\Delta\|_F \leq \frac{\sqrt{2 \lambda_{\max}} \Phi^{-1} \left( 1 - \delta/2 \right)}{\lambda_{min}(\hat \Sigma)(\lambda_{min}(\hat \Sigma) - \sqrt{2 \lambda_{\max}} \Phi^{-1} \left( 1 - \delta/2 \right))}
\end{align}


Therefore, using the notation $\Theta = \Sigma^{-1}$,

\begin{align}
    \|\hat \Theta - \Theta \|_2 \leq \frac{\sqrt{2 \lambda_{\max}} \Phi^{-1} \left( 1 - \delta/2 \right)}{\lambda_{min}(\hat \Sigma)(\lambda_{min}(\hat \Sigma) - \sqrt{2 \lambda_{\max}} \Phi^{-1} \left( 1 - \delta/2 \right))}
    \label{eq:theta_l2_norm}
\end{align}

\subsection{Rate of convergence}

We have that
\begin{align}
    \varepsilon = \sqrt{2 \lambda_{\max}} \Phi^{-1} \left( 1 - \delta/2 \right)
\end{align}
decreases as $\mathcal{O}(n^{-1/2})$, so asymptotically we have that 
\begin{align}
    \|\hat \Theta - \Theta \|_2 \leq \frac{c n^{-1/2}}{C - n^{-1/2}}
\end{align}
for some constants $c>0$ and $C>0$.  The r.h.s.\ is still $\mathcal{O}(n^{-1/2})$, which can be seen by letting $m$ be large enough such that $m^{-1/2}<C$, and noting that the constant $\frac{1}{C-m^{-1/2}} > \frac{1}{C-n^{-1/2}}$ for all $n>m$.

\subsection{Statistical test}

Algorithm~\ref{aistats2021:algo:test} summarizes the steps needed to test conditional dependence. The null-hypothesis of the test is that $\Theta_{ij} = 0$. We note that its acceptance does not imply conditional independence for all distributions $P_X$, and rejection of the null hypothesis implies dependence between the $i^{th}$ and $j^{th}$ variables conditioned on all other variables.

\begin{algorithm}[ht]
\caption{Proposed statistical test.
}\label{aistats2021:algo:test}
\begin{algorithmic}[1]
\Require{$\delta$, the significance level of the test; 
$\operatorname{X} = (X_1, ..., X_n)$ the set of random variables of dimension $p$ with sample size $n$.}

    
    \State Compute $\hat{\Sigma}$, the unbiased estimator of $\Sigma$ from data.
    
    \State Compute $\hat{\Theta} = \hat{\Sigma}^{-1}$, the estimator of the precision matrix.
    
    \State Compute $\lambda_{max}$ and $\lambda_{min}$, the largest and smallest eigenvalue of $\Cov(\hat{\Sigma})$, respectively.
    
    \State Compute the error bound $\varepsilon = \sqrt{2 \lambda_{max}} \Phi^{-1} \left( 1 - \delta/2 \right)$, where $\Phi$ is the CDF of a standard normal distribution. 
    
    
    
    
    
    \State Compute a $threshold$ as the r.h.s. of Equation~\eqref{eq:theta_l2_norm}.
    
    \State Reject the null hypothesis if $\hat{\Theta}_{ij} > threshold$.
\end{algorithmic}
\label{EJS:algo:hypothesis_testing_threshold}
\end{algorithm}

\section{Experiments}
We show empirical results for producing confidence intervals of the eigendecomposition and the precision matrix on both synthetic and real-world data from the medical and physics domain, described below. In addtion, we also use synthetic data to demonstrate the performance of the developed statistical test.

\subsection{Datasets}
\paragraph{Synthetic data}
We generated synthetic datasets of dimension $p$  from Normal and Laplace distribution for some precision matrix $\Theta$ following \cite{kotz2001asymmetric}. We generate random precision matrices as follows: Firstly, we generate a random unnormalized affinity matrix $A$ of size $p\times p$. Subsequently, we perform its eigendecomposition, and ensure positive eigenvalues. After assembling this decomposition with the adjusted eigenvalues back, we obtain a positive definite matrix, which we consider to be $\Theta$. In our experiments, we generated precision matrices with at least one zero.

\paragraph{Knee osteoarthritis}

We used the data from baselines of two publicly available follow-up patient cohorts -- Osteoarthritis Initiative (OAI, $n_{subjects}=4796$), and Multicenter Osteoarthritis Study (MOST, $n_{subjects}=3026$). We applied our method for conducting an analysis for assessing the dependence between the symptoms and radiographic progression of knee osteoarthritis in the medial side of the joint. 

We leveraged the gradings done according to the Osteoarthritis Research Society International (OARSI) grading atlas~\citep{altman2007atlas}. The following variables from the radiographic part of the OAI and MOST datasets were used: 
\begin{enumerate}[nosep]
    \item Osteophytes (bone spures) severity in the tibia bone (OSTM).
    \item Ostophytes severity in the femur (OSFM).
\end{enumerate}

We also included the symptomatic assessments done according to the Western Ontario and McMaster
Universities Osteoarthritis Index (WOMAC)~\citep{bellamy1995womac}. WOMAC allows for quantification of the patient's pain, and we included the total WOMAC score, calculated as a sum of all the scores of its subsections. 

\paragraph{Higgs boson}

We consider the Higgs boson classification dataset from \cite{baldi2014searching}, containing 11 million observations. 
For this experiment we used the seven high-level features derived from the kinematic properties measured for the decay products after a particle collision occurs. Four of them, i.e., $ m_{jj}$, $m_{jjj}$, $m_{lv}$ and $m_{jlv}$, involve the observable decay products leptons $l$ and jets, and the remaining three ($m_{bb}$, $m_{wbb}$ and $m_{wwbb}$) are related to the generation of the Higgs bosons.

\subsection{Bounds on the eigendecomposition}

We utilize the theoretical results presented above to
bound the eigendecomposition of the covariance matrix for the two described experiments. The data is whitened as a pre-processing step.
\Cref{fig:eigvals_oa,fig:eigvals_higgs_third} show that the obtained upper and lower bound on the empirical eigendecomposition are valid.

\begin{figure}
    \centering
    \subfloat[Lower bound]{
    \includegraphics[width=0.3\linewidth]{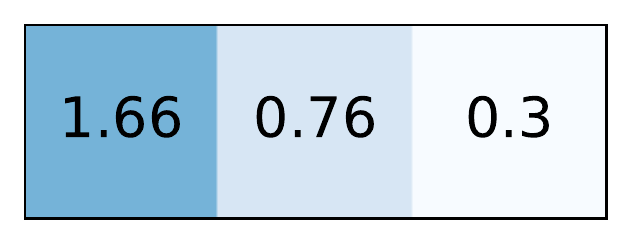}
    } \hfill
    \subfloat[Empirical]{
    \includegraphics[width=0.3\linewidth]{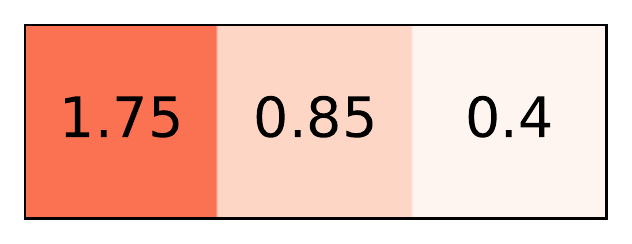}
    } \hfill
    \subfloat[Upper bound]{
    \includegraphics[width=0.3\linewidth]{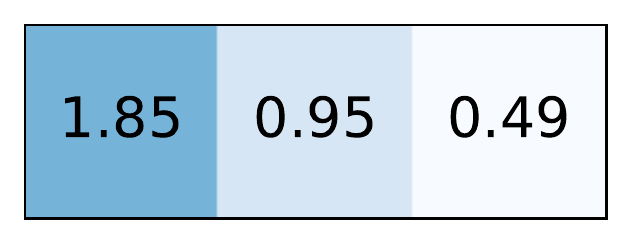}
    } \hfill
    \par
    \subfloat[Lower bound]{
    \includegraphics[width=0.3\linewidth]{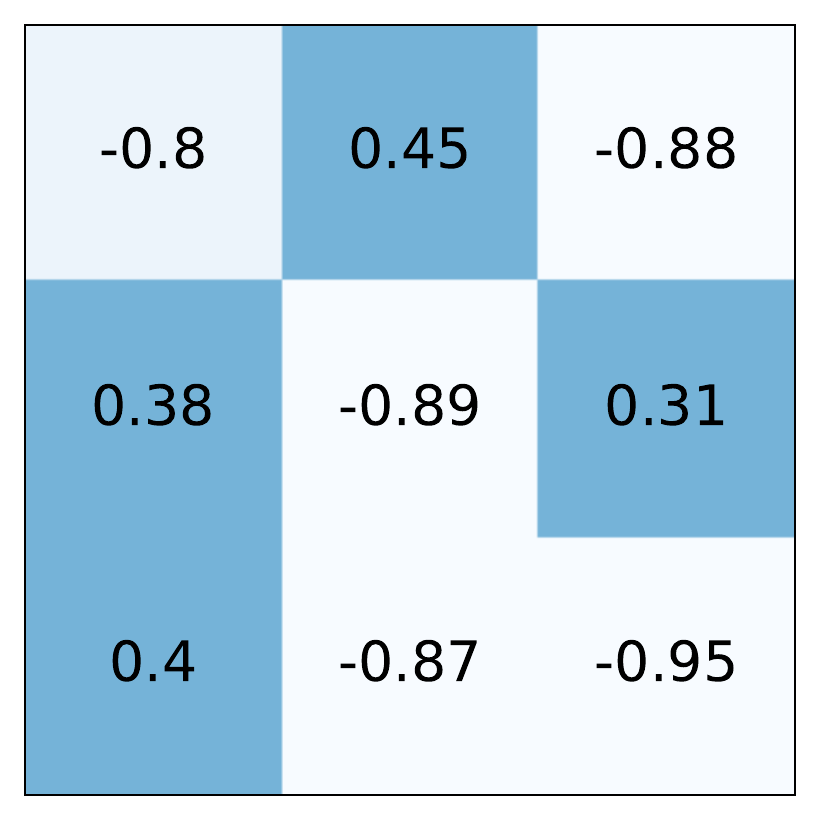}
    } \hfill
    \subfloat[Empirical]{
    \includegraphics[width=0.3\linewidth]{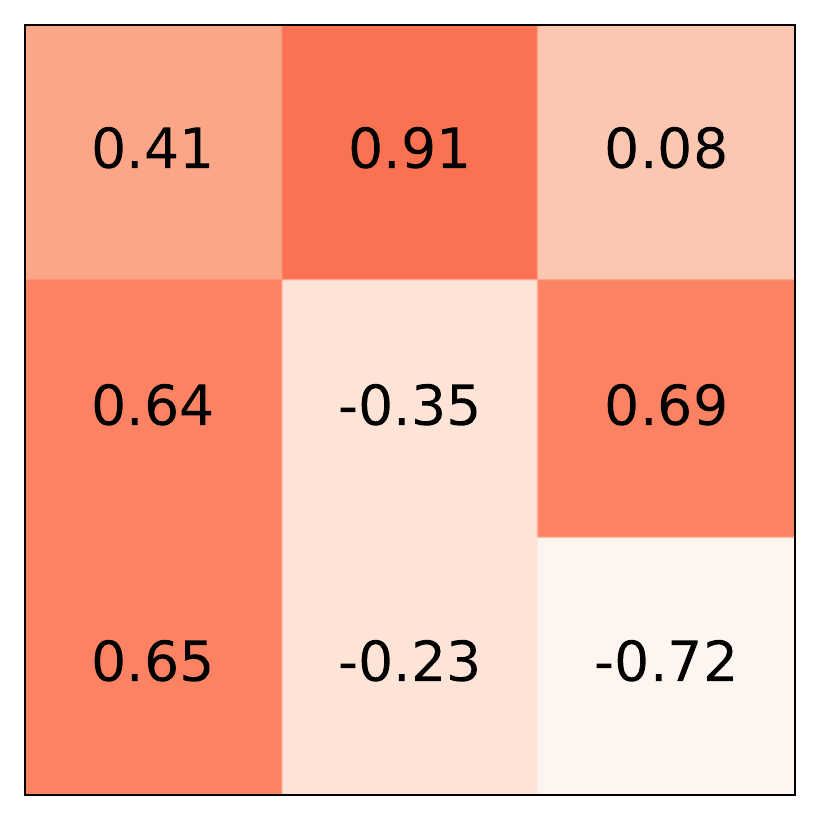}
    } \hfill
    \subfloat[Upper bound]{
    \includegraphics[width=0.3\linewidth]{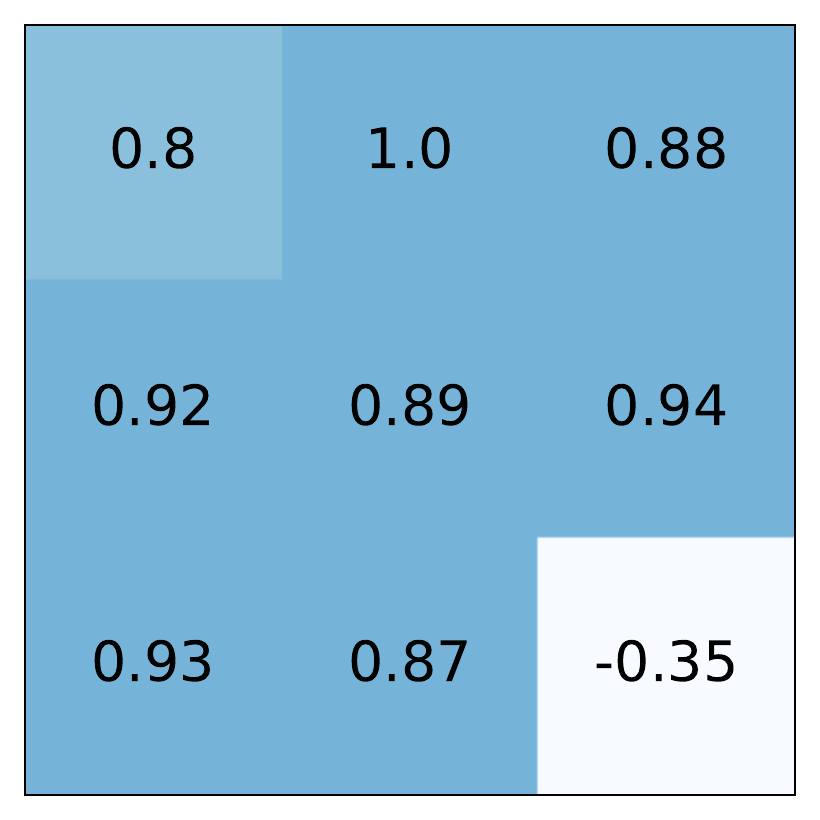}
    } \hfill
    \caption{Bounds on the eigenvalues (first row) and eigenvectors (second row) for the Osteoarthritis data [WOMAC, OSTM, OSFM]}
    \label{fig:eigvals_oa}
\end{figure}

\paragraph{Application: Interpretation of the medical data} In the case of the knee OA data (\autoref{fig:eigvals_oa}), we observe the following relationships known in the medical literature using the eigenvectors and the bounds: 
\begin{enumerate}
    \item Osteoarthritis develops in both tibia and femur for most of the cases. This can be seen from the first eigenvector. Our method yields non-trivial bounds for the imaging features, and returns high uncertainty for the symptoms.
    \item Symptoms (WOMAC) have very limited association with structural features of the disease (osteophytes denoted by OSTM and OSFM features). Here, we consider the second eigenvector: a non-trivial bound is obtained for the symptoms, and high uncertainty is seen for the imaging features.
    \item There exist some small number of cases in the data, for which tibial and femoral medial OA do not develop simultaneously. We can see that the empirical value for the symptoms in the case of the third eigenvector is nearly zero, and it has high uncertainty. The imaging features, in contrast, have rather tight bounds of the same size.
\end{enumerate}

\paragraph{Scalability test}
The experiment with the Higgs boson dataset (\autoref{fig:eigvals_higgs_third})  shows that we can get non-trivial bounds on data with several million examples and multiple features. Furthermore, our unoptimized implementation of the method runs several orders of magnitude faster than a bootstrap procedure. We run all experiments on a CPU of a regular laptop.

\begin{figure}
    \centering
    \subfloat[Lower bound]{
    \includegraphics[width=0.3\linewidth]{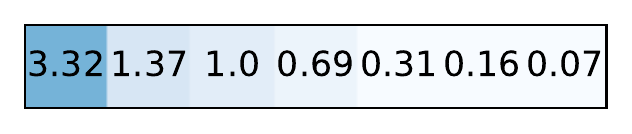}
    } \hfill
    \subfloat[Empirical]{
    \includegraphics[width=0.3\linewidth]{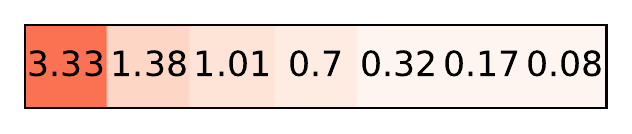}
    } \hfill
    \subfloat[Upper bound]{
    \includegraphics[width=0.3\linewidth]{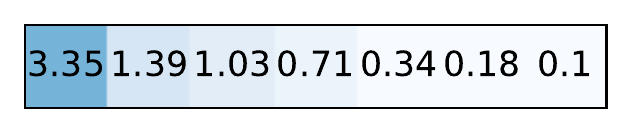}
    } \hfill
    \par
    \subfloat[Lower bound]{
    \includegraphics[width=0.3\linewidth]{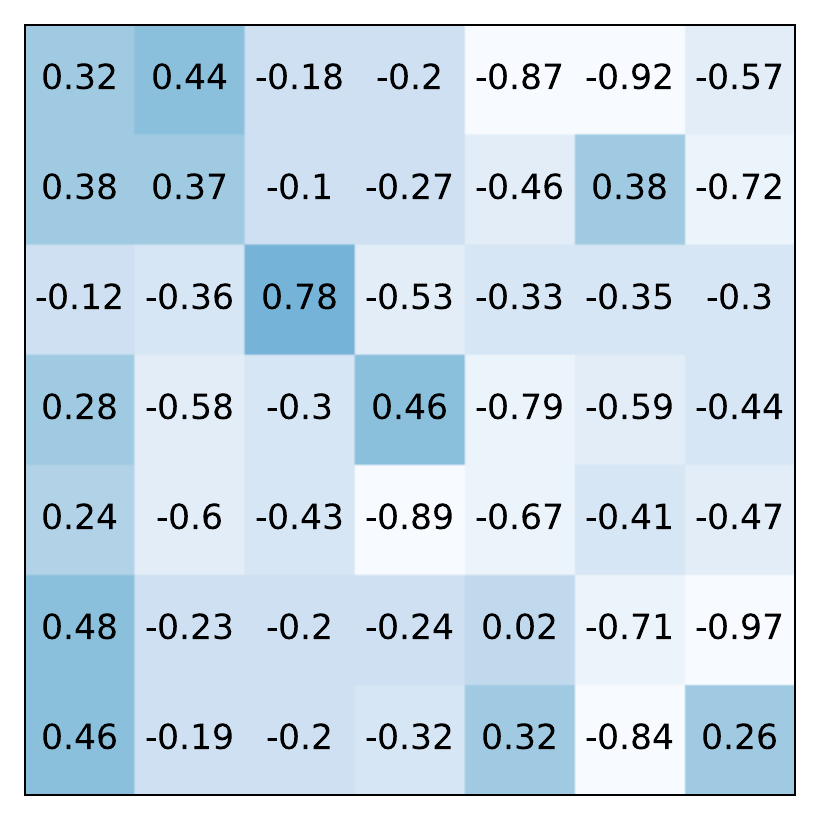}
    } \hfill
    \subfloat[Empirical]{
    \includegraphics[width=0.3\linewidth]{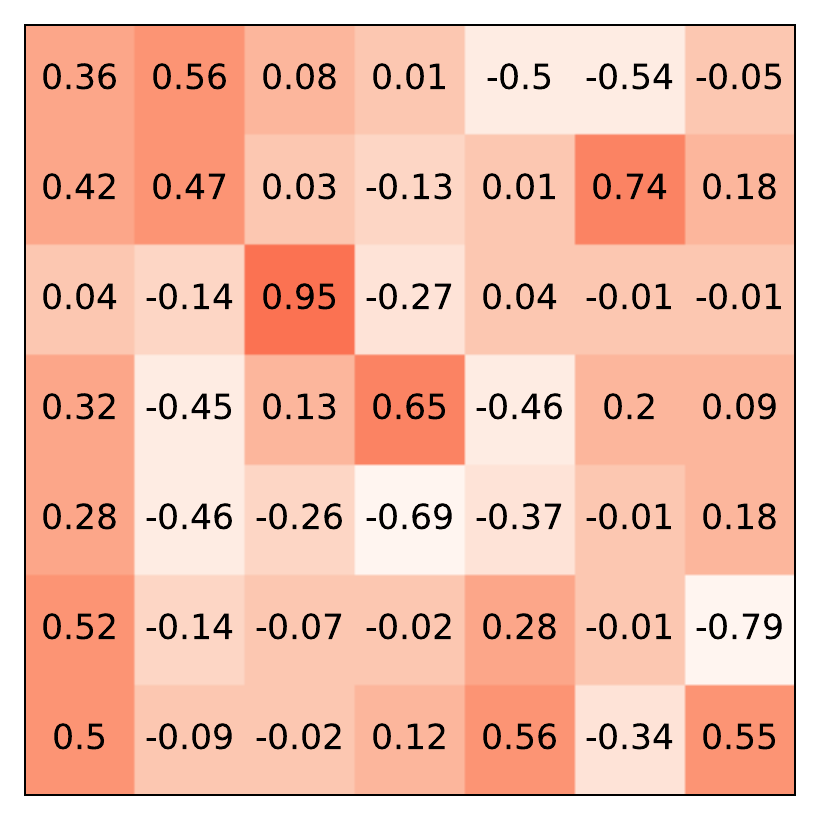}
    } \hfill
    \subfloat[Upper bound]{
    \includegraphics[width=0.3\linewidth]{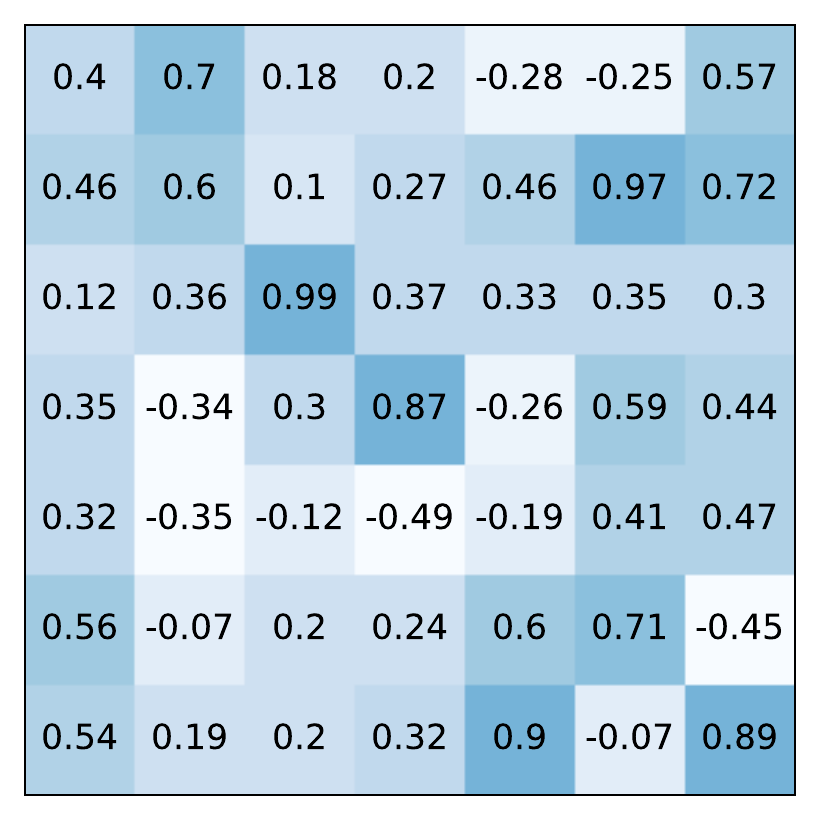}
    } \hfill
    \caption{Bounds on the eigenvalues (first row) and eigenvectors (second row) for the Higgs dataset [$ m_{jj}$, $m_{jjj}$, $m_{lv}$, $m_{jlv}$, $m_{bb}$, $m_{wbb}$ and $m_{wwbb}$].}
    \label{fig:eigvals_higgs_third}
\end{figure}

\subsection{Bounds on the precision matrix}
We make use of the theoretical result in Section~\ref{section:pm_pertubration_bounds} to bound the entries of the precision matrix of the two datasets.
The lower and upper bounds on the precision matrix for the knee osteoarthritis data is shown in Figure~\ref{fig:pmatrix_oa}.
Figure~\ref{fig:pmatrix_higgs_third} illustrates the confidence intervals on the precision matrix obtained from the Higgs dataset.


\begin{figure}[ht!]
    \centering
    \subfloat[Lower bound]{
    \label{fig:pmatrix_oa:a}\includegraphics[width=0.3\linewidth]{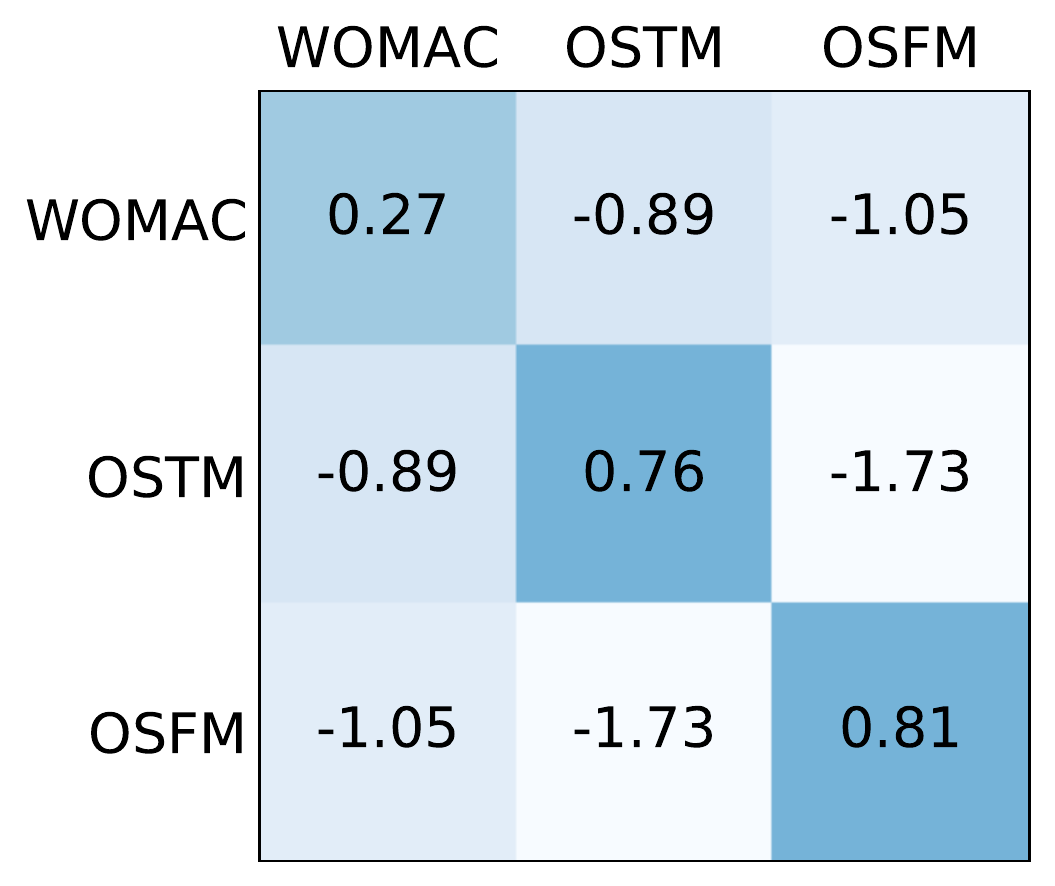}
    } \hfill
    \subfloat[Empirical]{
    \label{fig:pmatrix_oa:b}\includegraphics[width=0.3\linewidth]{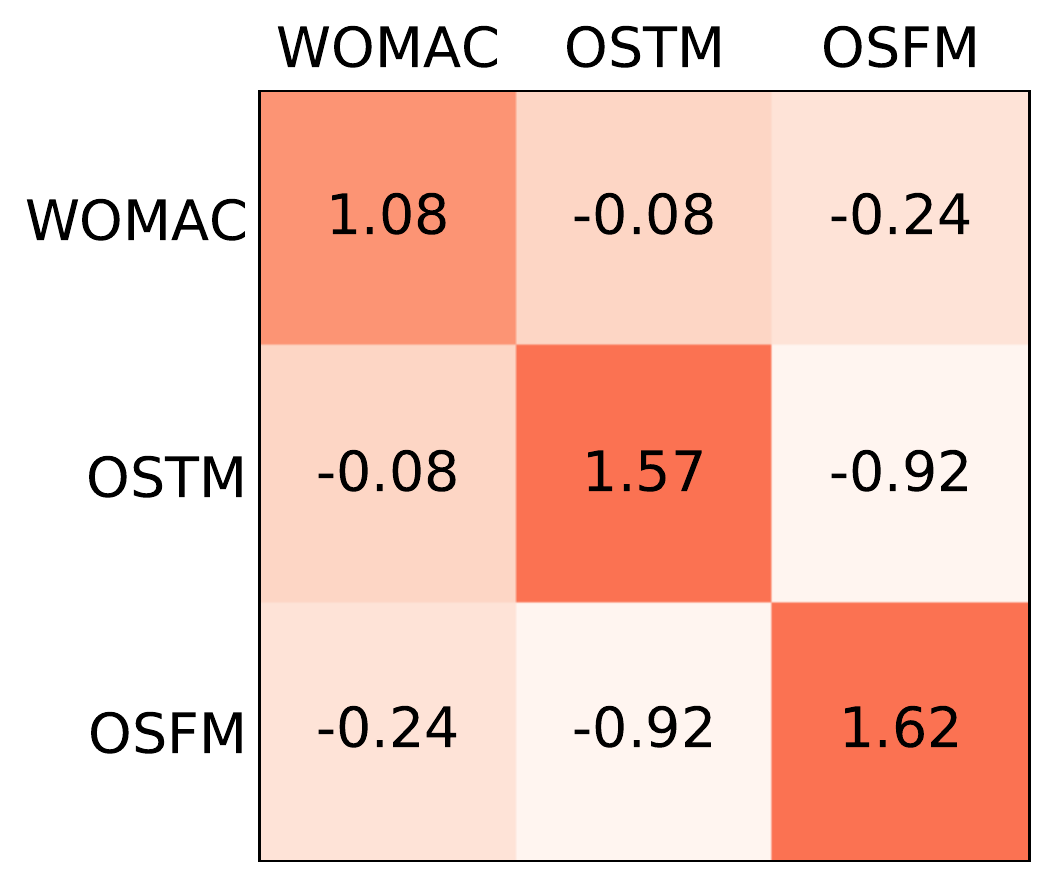}
    } \hfill
    \subfloat[Upper bound]{
    \label{fig:pmatrix_oa:c}\includegraphics[width=0.3\linewidth]{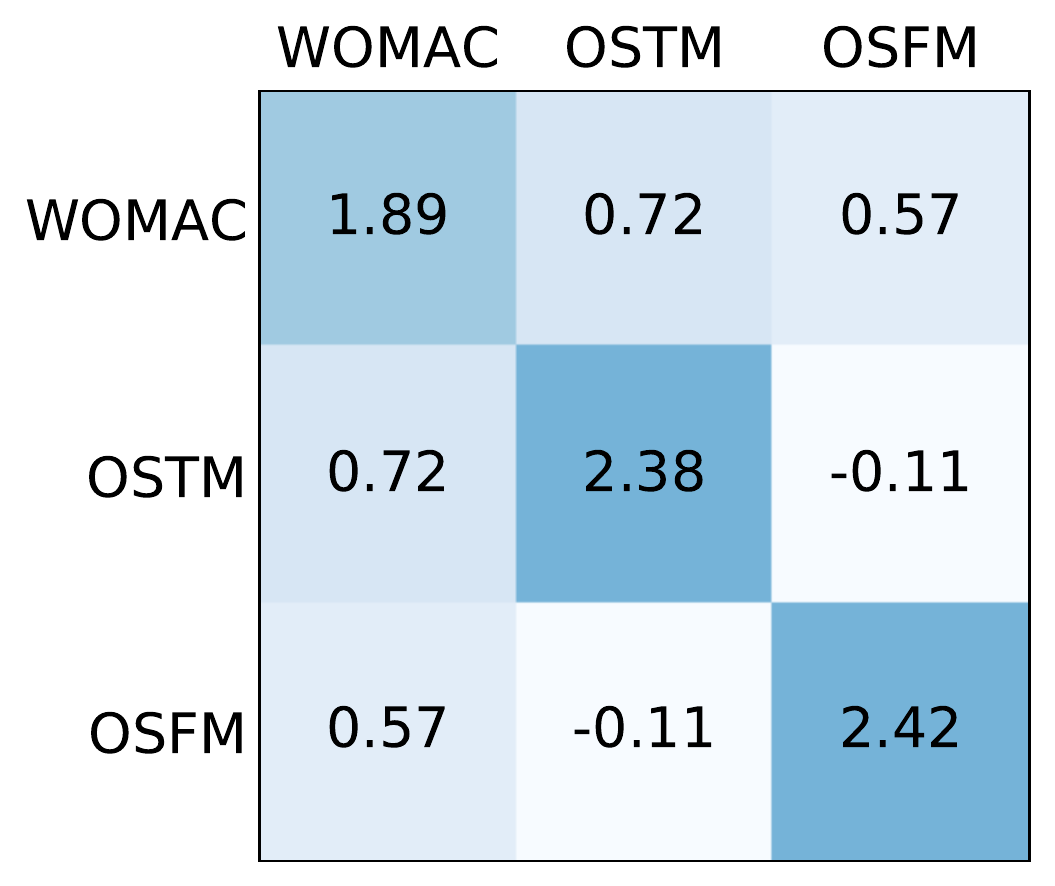}
    } \hfill
    \caption{Bounds on the precision matrix for the Osteoarthritis data [WOMAC, OSTM, OSFM]
    }
    \label{fig:pmatrix_oa}
\end{figure}

\begin{figure}[ht!]
    \centering
    \subfloat[Lower bound]{
    \label{fig:pmatrix_higgs_third:a}\includegraphics[width=0.3\linewidth]{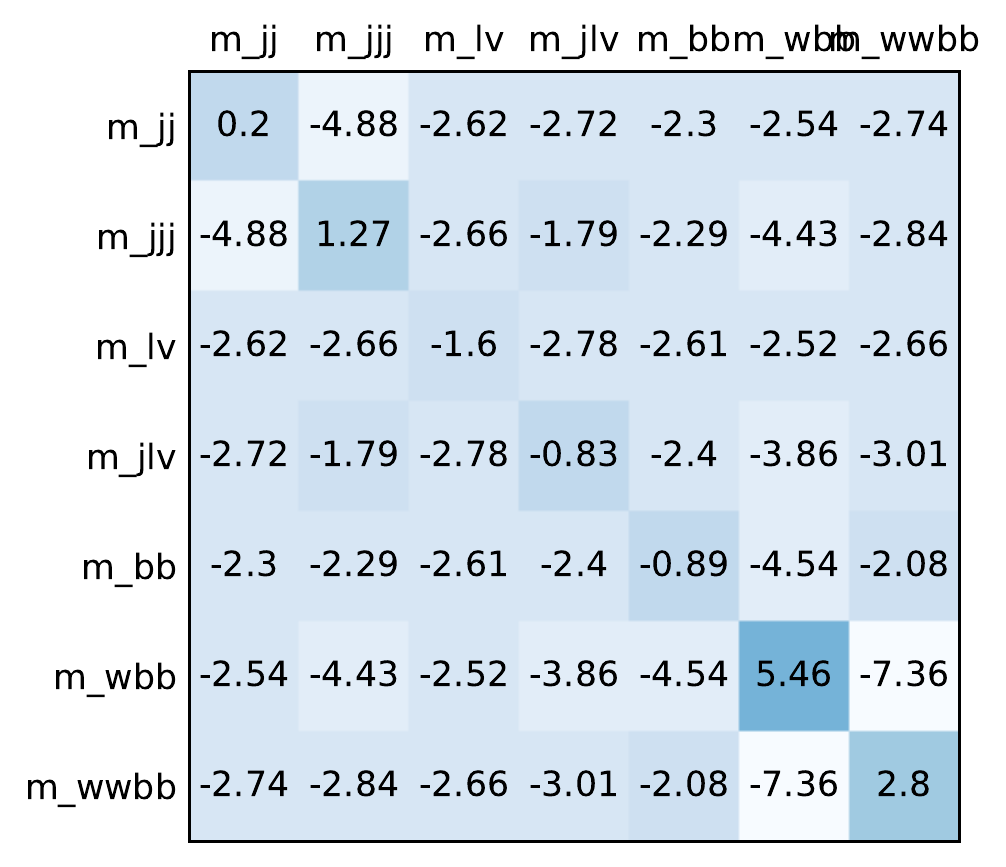}
    } \hfill
    \subfloat[Empirical]{
    \label{fig:pmatrix_higgs_third:b}\includegraphics[width=0.3\linewidth]{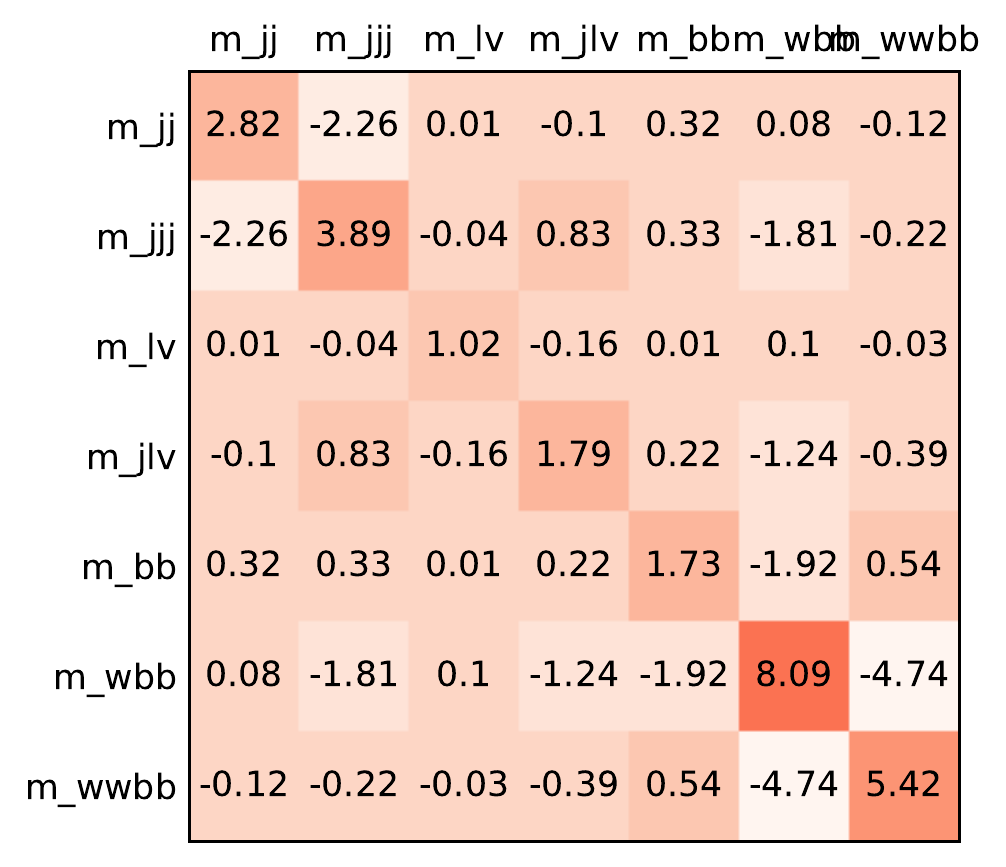}
    } \hfill
    \subfloat[Upper bound]{
    \label{fig:pmatrix_higgs_third:c}\includegraphics[width=0.3\linewidth]{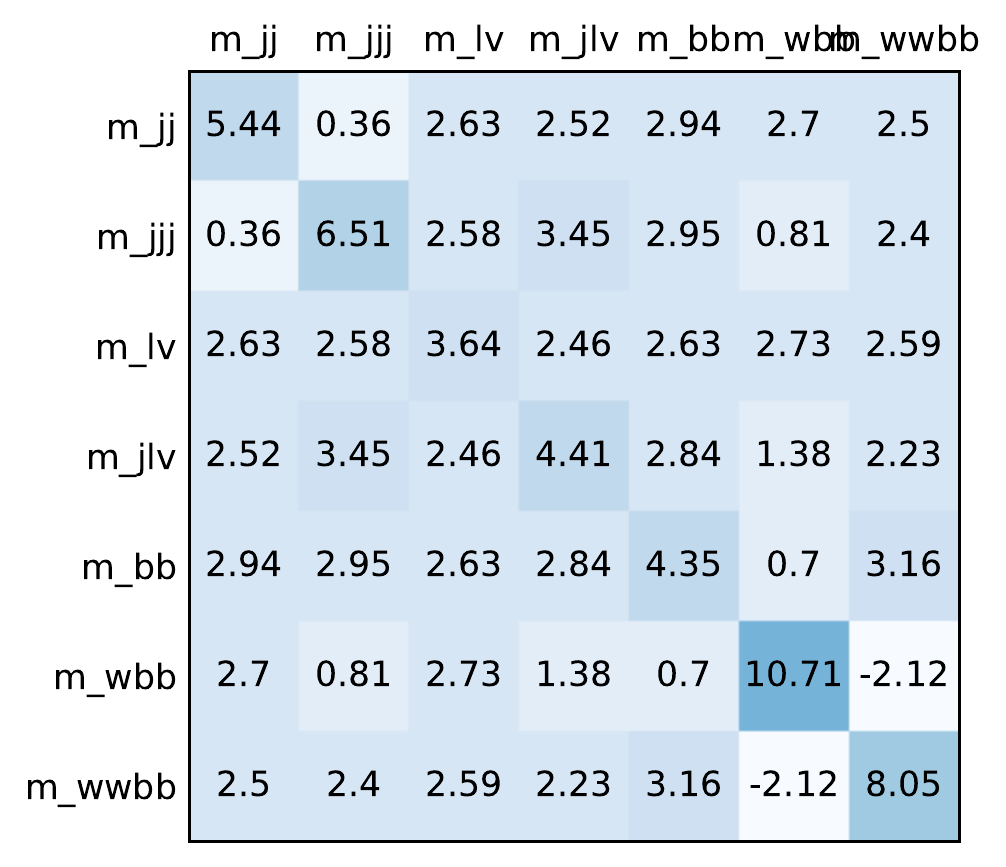}
    } \hfill
    \caption{Bounds on the precision matrix for Higgs data [$ m_{jj}$, $m_{jjj}$, $m_{lv}$, $m_{jlv}$, $m_{bb}$, $m_{wbb}$ and $m_{wwbb}$] 
    }
    \label{fig:pmatrix_higgs_third}
\end{figure}

\subsection{Statistical testing}
Figure \ref{fig:fisher} visually illustrates the performance of the developed statistical test and its comparison to the Fisher's transform. Our results show that the Fisher's transform-based test is well calibrated for the Gaussian data, but yields false positives for Laplace data. In contrast, our test is always conservative regardless of the data distribution, and therefore yields sound conclusions on statistical significance.

\begin{figure}[ht]
    \centering
    \subfloat[Gaussian distribution]{
     \includegraphics[width=0.45\linewidth]{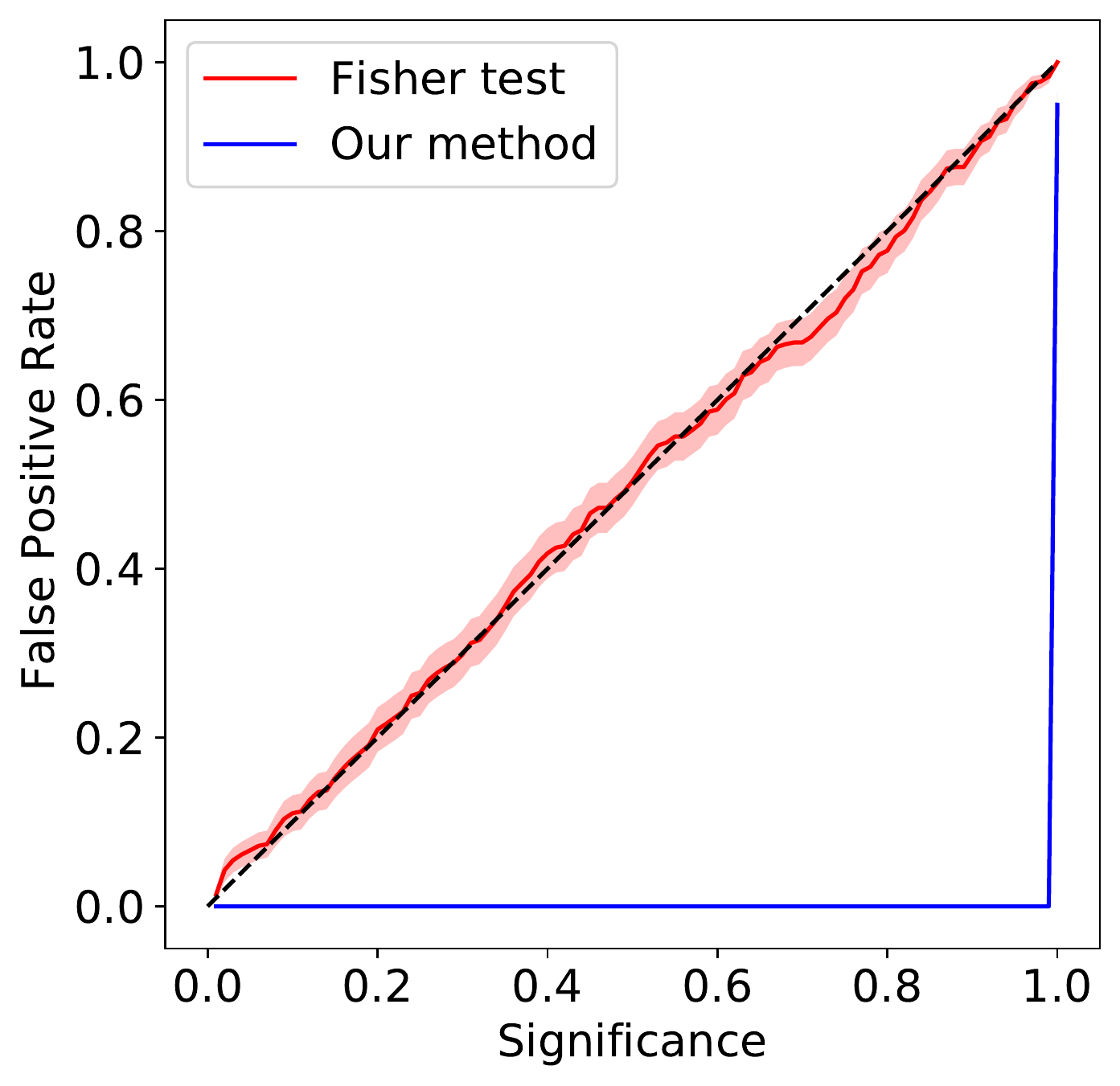} \hfill
    }
    \subfloat[Laplace distribition]{
    \includegraphics[width=0.45\linewidth]{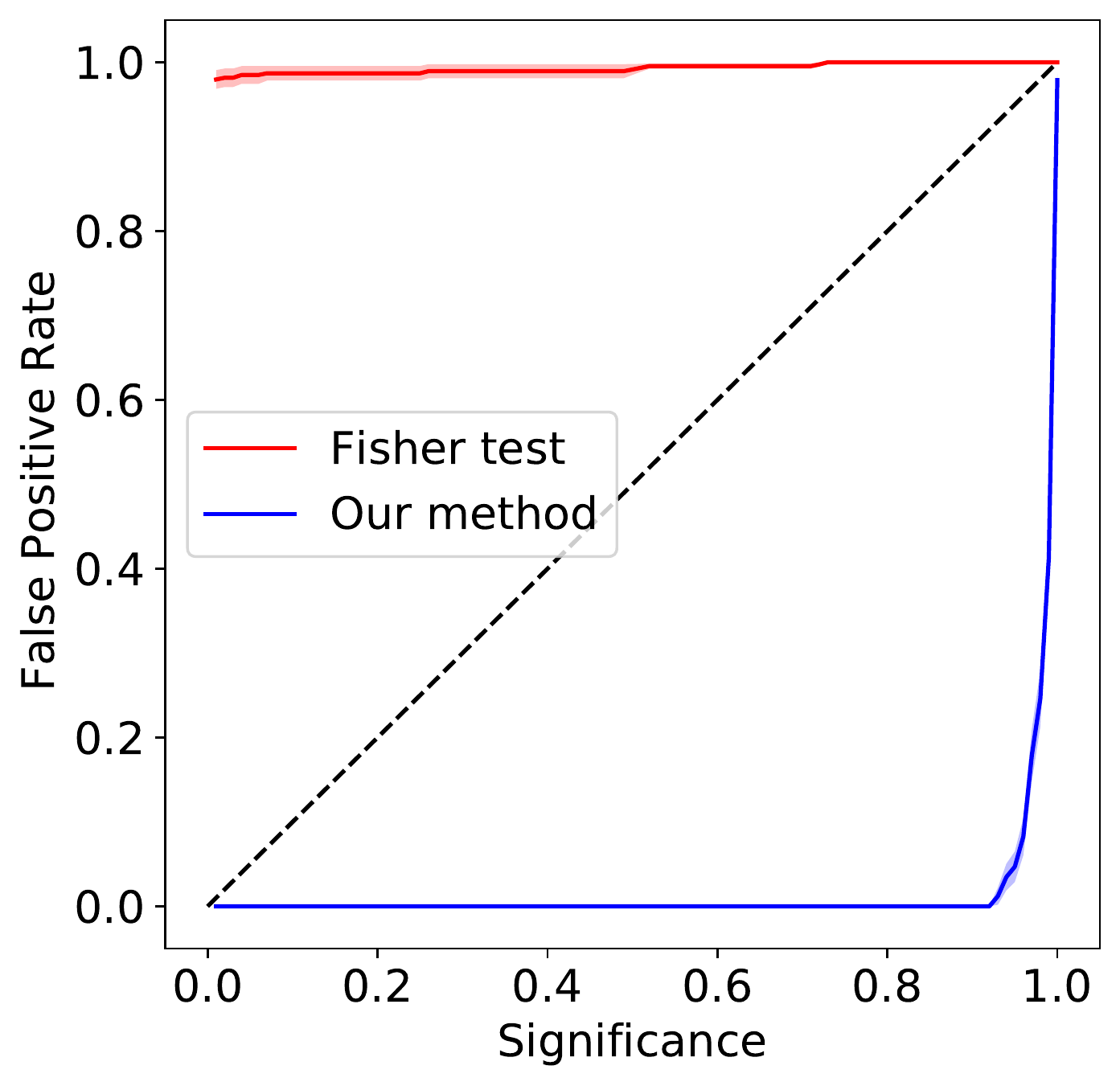}
    }
    \caption{Graphical comparison between the Fisher's test and our method on Gaussian and Laplace data. $n=500000$, $p=5$, simulation repeated $100$ times.}
    \label{fig:fisher}
\end{figure}

Figure~\ref{fig:graphs} visually illustrates the (inferred) graphical structure of synthetic data sampled from a Laplace distribution. Figure~\ref{fig:graphs:ground_truth} represents the ground truth graphical model and \Cref{fig:graphs:fisher,fig:graphs:our} represent the inferred structure using the Fisher's test and our test, respectively. The tests are performed with a 95\% confidence level. 

\begin{figure}[ht!]
    \centering
    \subfloat[Ground truth]{
    \label{fig:graphs:ground_truth}\includegraphics[width=0.3\linewidth]{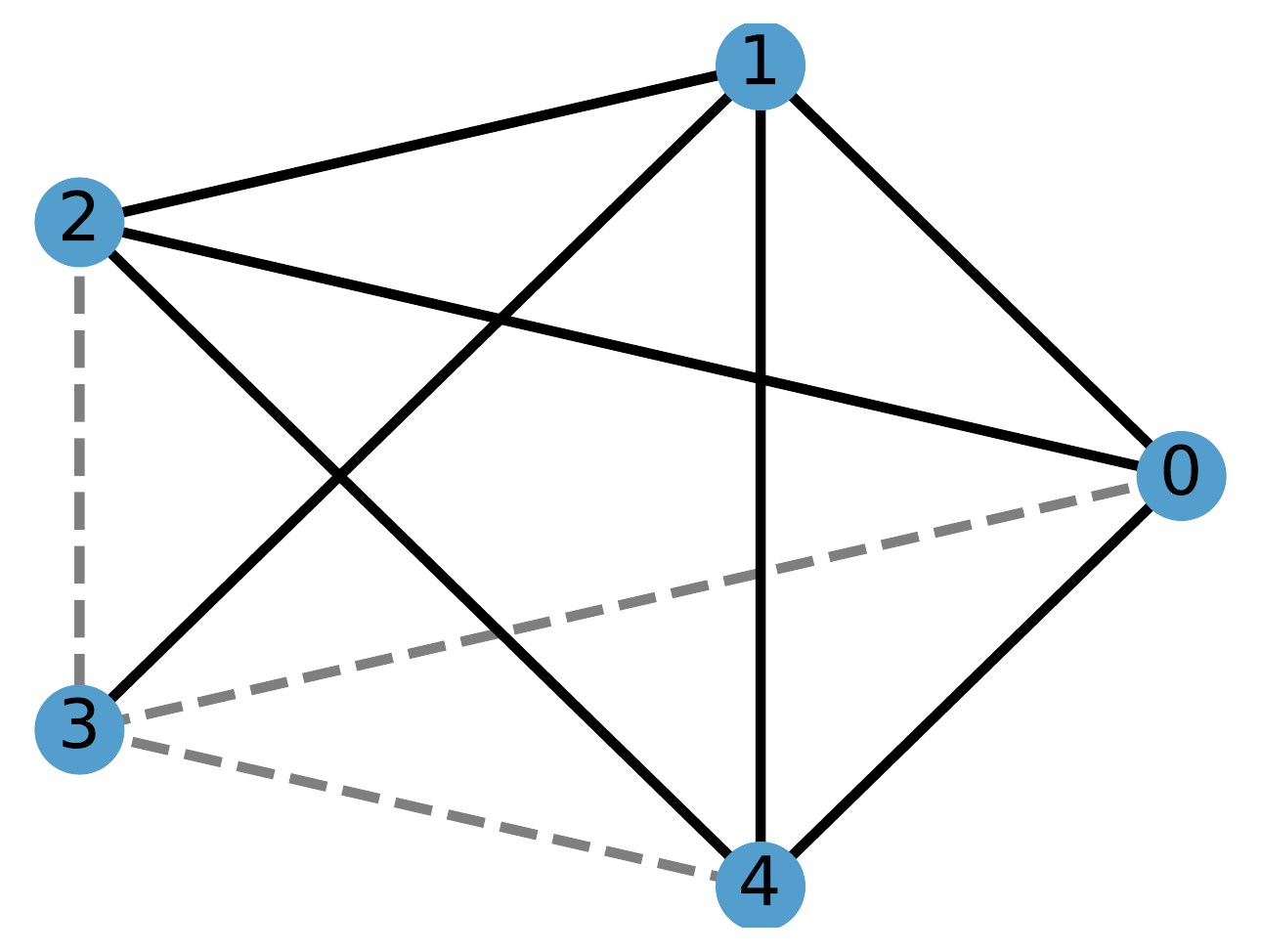}}
    \subfloat[Fisher's test]{
    \label{fig:graphs:fisher}\includegraphics[width=0.3\linewidth]{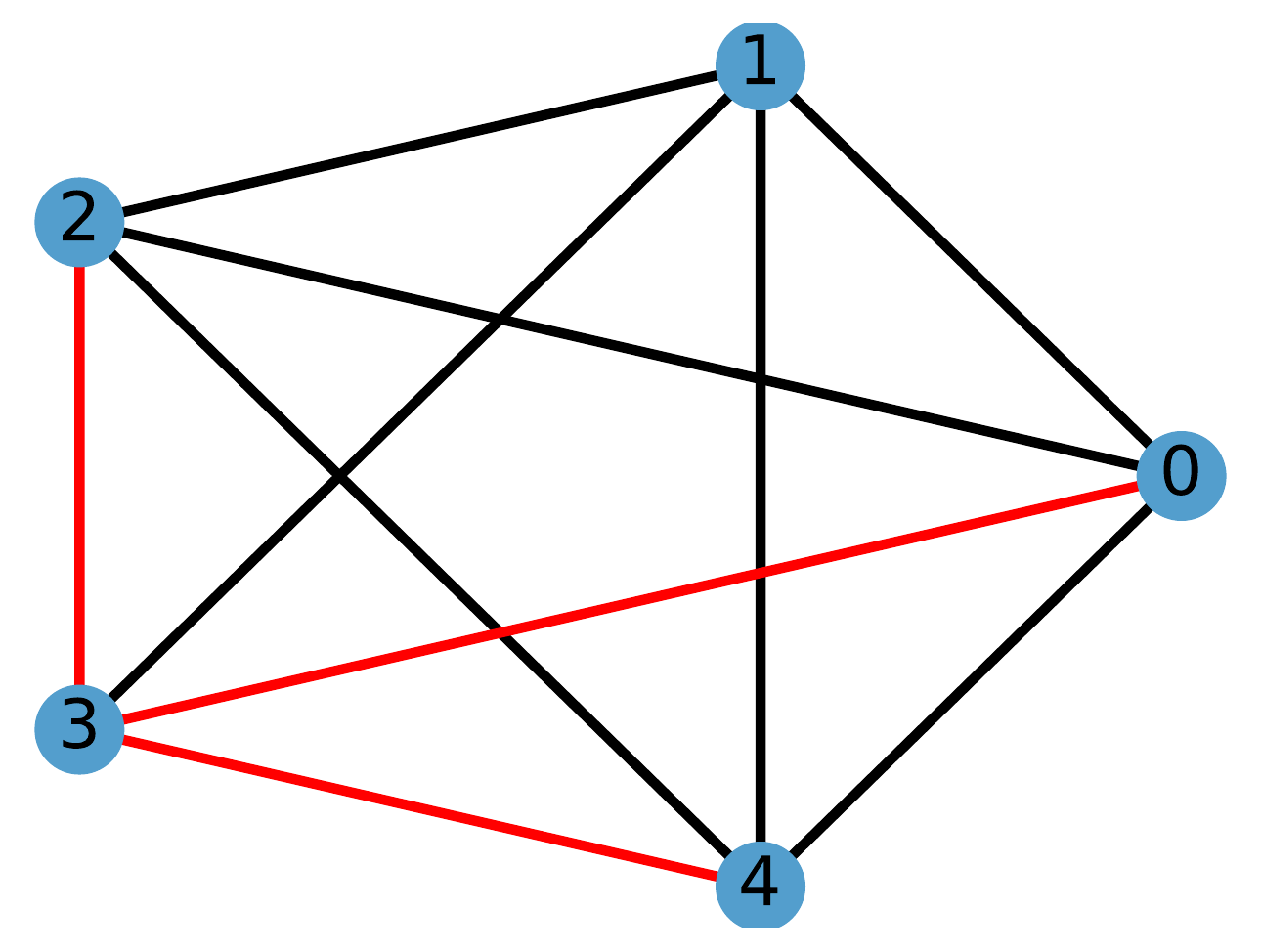}}
        \subfloat[Our test]{
    \label{fig:graphs:our}\includegraphics[width=0.3\linewidth]{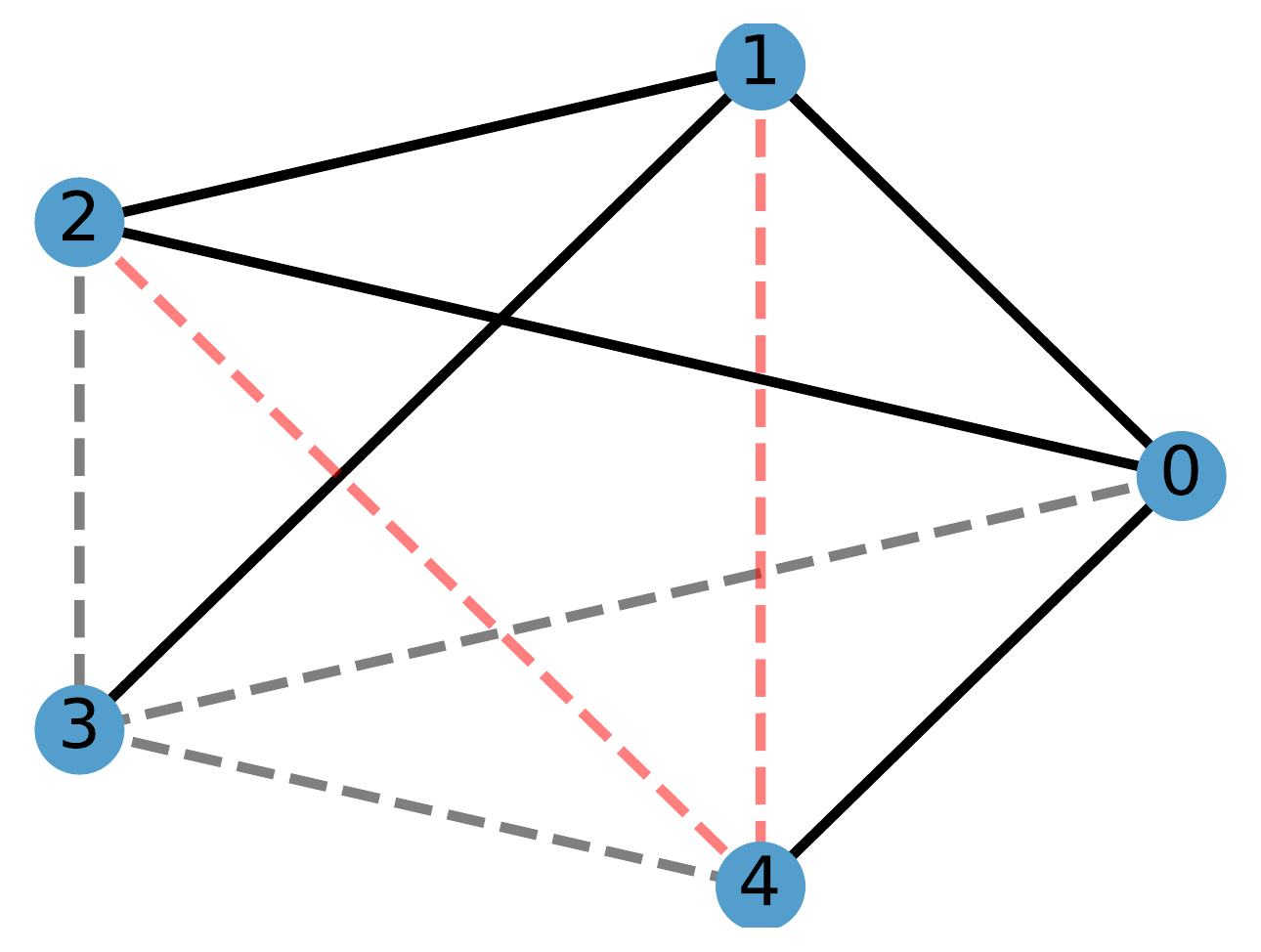}}
    \caption{Left: Ground truth graphical model of Laplace-dsitributed synthetic data with $n$=50000000 and $p$=5.  Middle: Inferred dependence structure using Fisher's test. Right: Inferred dependence structure using our test. Solid edges represent significant non-zero correlations (rejecting the null hypothesis). A dashed edge corresponds to not rejecting the null hypothesis. Wrongly inferred edges are marked in red.}
    \label{fig:graphs}
\end{figure}

\section{Discussion and conclusions}

In this paper, we have presented a novel methodology for obtaining confidence intervals on the eigendecomposition of covariance matrices, and a procedure for bounding the entries of precision matrices. Both methods first utilize the theory of U-statistics to bound the perturbation of the empirical covariance matrix. We then use Weyl’s theorem to obtain bounds on the eigenvalues and finally, we leverage the recently formalized eigenvalue-eigenvector identity to  bound the eigenvectors. To the best of our knowledge, this is the first work that explores the use of the eigenvalue-eigenvector identity in the context of estimating error bounds on the eigendecomposition of covariance matrices and confidence intervals on precision matrices. The confidence intervals on the precision matrix are derived using general pertubration bounds.

The main strength of the proposed methodology is that it does not impose any restrictions on the data distribution and on the existence of finite higher-order moments. In fact, the only assumptions are that i) the covariance matrix exists and ii) an unbiased estimator converges to it. We showed an application of the method for obtaining error bars on the precision matrix and proposed a procedure for detecting significantly non-zero partial correlations. The conducted experiments with both synthetic and real data verify the validity of the confidence intervals on the eigendecomposition and the precision matrix obtained with our proposed method, given a sufficient number of samples. 

Despite the benefits of the method, this paper has still some limitations. As such, the main downside of our approach is that, in practice, the number of samples has to be exponentially larger than the dimensionality of the data in order to obtain non-trivial bounds. Secondly, we did not fully explore the possibilities for tightening the bounds on the eigenvectors, however, we proposed orthonormality constraints-based approach in Section~\ref{aistats2021:sec:orthonormality}.
However, we note that the method is consistent for all covariances with non-repeated eigenvalues, though possibly with poor convergence in the event that eigenvalues are closely spaced.  This requirement is a direct consequence of the use of the eigenvalue-eigenvector identity, which is trivial in the case of repeated eigenvalues (Equation~\eqref{eq:EigenvalueEigenvectorIdentity} degenerates to $0=0$).

To conclude, we have proposed, to the best of our knowledge, the first application of the recently introduced eigenvalue-eigenvector identity to the estimation of confidence intervals over eigenvectors, resulting in a consistent estimator with computation linear in the number of samples.  Furthermore, we derived a method to bound the precision matrix and we demonstrated its practical application to real-world datasets, including particle physics and medicine. 
Source code is available at~\url{https://github.com/tpopordanoska/confidence-intervals}. 

\appendix

\section{Proof of Theorem~\ref{aistats2021:theorem:covariance_covariance_ustatistic}: Derivation of the covariance of the U-statistics for the covariance matrix} \label{aistats2021:sec:appendix_proof_covcov}

\subsection{Overview}
In this appendix, we show the details of the derivation of Theorem~\ref{aistats2021:theorem:covariance_covariance_ustatistic}. We derive minimum-variance unbiased estimates of the covariance between two $U$-statistics estimates $\hat{\Sigma}_{ij}$ and $\hat{\Sigma}_{kl}$, where $(i,j,k,l)$ range over each of the $d$ variates in a covariance matrix $\hat{\Sigma}$. We note $h$ and $g$ the corresponding kernel of order 2 for $\hat{\Sigma}_{ij}$ and  $\hat{\Sigma}_{kl}$, where
\begin{align}
    h(u_1,u_2) &= \dfrac{1}{2} \left( X_{i_1} - X_{i_2} \right)  \left( X_{j_1} - X_{j_2} \right), \mbox{with } u_r=(X_{i_r},X_{j_r})^T \\
    g(v_1,v_2) &= \dfrac{1}{2} \left( X_{k_1} - X_{k_2} \right)  \left( X_{l_1} - X_{l_2} \right), \mbox{with } v_r=(X_{k_r},X_{l_r})^T.
\end{align}
Then, the covariance $\Cov(\hat{\Sigma}_{ij},\hat{\Sigma}_{kl}) $ for the two $U$-statistics $\hat{\Sigma}_{ij}$ and $\hat{\Sigma}_{kl}$ is 
\begin{align} 
\label{aistats2021:appendix:eq:CovCov}
    \Cov(\hat{\Sigma}_{ij},\hat{\Sigma}_{kl}) 
    &= \binom{n}{2}^{-1} \left( 2 (n-2) \zeta_1 + \zeta_2 \right) \\
    &= \binom{n}{2}^{-1} \left( 2 (n-2) \zeta_1  \right) + \mathcal{O}(n^{-2}) \nonumber
\end{align}
where $\zeta_1 = \Cov \left(  \E_{u_2}[h(u_1,u_2)], \E_{v_2}[g(v_1,v_2)] \right) $.

Depending on the equality and inequality of these four index variables, the empirical covariance estimate takes a different kernel form.  We have employed a computer assisted proof to determine that there are seven different forms and that each of the unique $\binom{p^2-\binom{p}{2}}{2}$ entries in $\operatorname{Cov}(\hat{\Sigma})$ (cf.\ Eq.~\eqref{aistats2021:eq:Covariance_of_Ustat_estimator} from the main text) can be mapped to one of these seven cases by a simple variable substitution.

In the sequel, we first describe the algorithm that determines the seven cases (Sec.~\ref{aistats2021:appendix:subsec:algorithm_forsevencases}), we derive empirical estimators for each of these seven cases (Sec.~\ref{aistats2021:appendix:subsec:derivation_sevencases}) and show that in all cases we have linear computation time in the number of samples (Sec.~\ref{aistats2021:appendix:subsec:proof_linear_computation_time}).

\subsection{Description of the algorithm providing the seven cases} 
\label{aistats2021:appendix:subsec:algorithm_forsevencases}

We formally describe the algorithm that provided us 7 cases for the derivation of $\Cov(\hat{\Sigma}_{ij},\hat{\Sigma}_{kl})$ of Theorem~\ref{aistats2021:theorem:covariance_covariance_ustatistic}, where $\left(i,j,k,l \right)$ vary over the set of $d$ variables. 

\begin{description}
    \item[\textbf{Enumeration}] 
    First, we enumerate all configurations of $\Cov(\hat{\Sigma}_{ij},\hat{\Sigma}_{kl})$, which can be encoded as a non-unique assignment matrix of variables $i,j,k,l$ to instantiated variables $\left(a,b,c,d \right)$. For a fixed assignment of $i$ to variable $a$, we can list all possible assignments of the 3 remaining variables  $\left(j,k,l\right)$ to any $\left(a,b,c,d\right)$.  Na\"{i}vely, we have $4^3$ possible assignments, but many of them will be equivalent by variable substitution.  To test whether two forms are equivalent, it is sufficient to test a reduced form for equality. 
    
    \item[\textbf{Reduced Form}] 
    We map a variable assignment to a reduced form by re-labeling variables sorted by the number of occurrences, which reduces the number of possible matches up-to non-uniqueness of the mapping due to equal numbers of variable occurrences.  This ambiguity is then resolved by testing for symmetries. 
    
    \item[\textbf{Symmetry}] 
    Symmetry of the covariance operator brings the following equally that we take into consideration in testing for equivalence:
    \begin{align}
    \Cov(\hat{\Sigma}_{ij},\hat{\Sigma}_{kl}) 
    &= \Cov(\hat{\Sigma}_{kl},\hat{\Sigma}_{ij}) = 
    \Cov(\hat{\Sigma}_{ij},\hat{\Sigma}_{lk}) =  
    \Cov(\hat{\Sigma}_{lk},\hat{\Sigma}_{ij})  \\  \nonumber  
    &= \Cov(\hat{\Sigma}_{lk},\hat{\Sigma}_{ji}) = 
    \Cov(\hat{\Sigma}_{ji},\hat{\Sigma}_{kl})  
    =\Cov(\hat{\Sigma}_{ji},\hat{\Sigma}_{lk})
    \end{align}
\end{description}
The algorithm outputs each variable assignment that is not equivalent by variable substitution to any previously enumerated assignment. 
The seven different cases are enumerated in Table~\ref{aistats2021:table:enumeration7cases}.

\begin{table}[htb] 
    \centering
        \begin{tabular}{|r|l|l|}
        \hline
        Cases & Indices & Correspondence \\
        \hline
        1 & $i \neq j,k,l$; $j \neq k,l$; $k \neq l$  & $\Cov(\hat{\Sigma}_{ij},\hat{\Sigma}_{kl})$  \\
        2 & $i=j$; $j \neq k,l$; $k = l$ &  $\Cov(\hat{\Sigma}_{ii},\hat{\Sigma}_{kk})$  \\
        3 & $i=j$; $j \neq k,l$; $k \neq l$ & $\Cov(\hat{\Sigma}_{ii},\hat{\Sigma}_{kl})$  \\
        4 & $i=k$; $j \neq i,k,l$; $k \neq l$ & $\Cov(\hat{\Sigma}_{ij},\hat{\Sigma}_{il})$  \\
        5 & $i=k$; $i \neq j$; $j=l$; & $\Var(\hat{\Sigma}_{ij})$  \\
        6 & $i=j=k$; $i \neq l$ & $\Cov(\hat{\Sigma}_{ii},\hat{\Sigma}_{il})$  \\
        7 & $i=j,k,l$ & $ \Var(\hat{\Sigma}_{ii})$ \\
        \hline
        \end{tabular}
    \caption{Enumeration and correspondence of the seven cases.} \label{aistats2021:table:enumeration7cases}
\end{table}

\subsection{The seven exhaustive cases} \label{aistats2021:appendix:subsec:derivation_sevencases}

We now derive linear-time finite-sample estimates of the covariance for each of the seven cases.

\paragraph{Notation}
We denote the $p$-dimensional data matrix with $n$ i.i.d samples as $X\in \mathbb{R}^{p\times n}$, data distribution as $P_X$, $X_i$ -- $i^{th}$ row of the data matrix, $\Xbar{X_iX_j} = \mathbb{E}_{X}[X_iX_j]$, $\Xbar{X_i}\myspace\Xbar{X_j} = \mathbb{E}_{X}[X_i]\mathbb{E}_{X}[X_j]$.


\subsubsection*{Case 1: $i \neq j,k,l$; $j \neq k,l$; $k \neq l$}

The kernels are 

\noindent\begin{minipage}{.5\linewidth}
\begin{align*}
&h(u_1,u_2) = \dfrac{1}{2} \left( X_{i_1} - X_{i_2} \right) \left( X_{j_1} - X_{j_2} \right); \\
&\E_{u_2}[h(u_1,u_2)] = \dfrac{1}{2} \left( X_{i_1} - \Xbar{X_i} \right) \left( X_{j_1} -\Xbar{X_j} \right);
\end{align*}
\end{minipage}\hfill
\begin{minipage}{.5\linewidth}
\begin{align*}
&g(v_1,v_2) = \dfrac{1}{2} \left( X_{k_1} - X_{k_2} \right)  \left( X_{l_1} - X_{l_2} \right) \\
&\E_{u_2}[g(v_1,v_2)] = \dfrac{1}{2} \left( X_{k_1} - \Xbar{X_k} \right)  \left( X_{l_1} - \Xbar{X_l} \right)
\end{align*}
\end{minipage}

\begin{align}
\zeta_1 &= \Cov \left[ 
\dfrac{1}{2} \left( X_{i_1} - \Xbar{X_i} \right) \left( X_{j_1} -\Xbar{X_j} \right), 
\dfrac{1}{2} \left( X_{k_1} - \Xbar{X_k} \right)  \left( X_{l_1} - \Xbar{X_l} \right)
\right] \\
&= \dfrac{1}{4} \biggl\{ \Cov \left[ 
X_{i_1} X_{j_1} - \Xbar{X_i} X_{j_1} - X_{i_1} \Xbar{X_j} ;
X_{k_1} X_{l_1} - \Xbar{X_k} X_{l_1} - X_{k_1}  \Xbar{X_l}
\right]
\biggr\} \nonumber \\
&= \dfrac{1}{4} \biggl\{ \E_{u_1} \big[ 
X_{i_1} X_{j_1} X_{k_1} X_{l_1} - \Xbar{X_i} X_{j_1} X_{k_1} X_{l_1} -  X_{i_1} \Xbar{X_j}  X_{k_1} X_{l_1} \nonumber \\
&\qquad\qquad\qquad - X_{i_1} X_{j_1} \Xbar{X_k} X_{l_1} + \Xbar{X_i} X_{j_1} \Xbar{X_k} X_{l_1} + X_{i_1} \Xbar{X_j} \myspace \Xbar{X_k}  X_{l_1} \nonumber \\
&\qquad\qquad\qquad - X_{i_1} X_{j_1} X_{k_1}  \Xbar{X_l} + \Xbar{X_i} X_{j_1} X_{k_1}  \Xbar{X_l} + X_{i_1} \Xbar{X_j} X_{k_1}  \Xbar{X_l} \big] \nonumber \\
& \qquad - \E_{u_1} \left[ X_{i_1} X_{j_1} - \Xbar{X_i} X_{j_1} - X_{i_1} \Xbar{X_j} \right] 
\E_{u_1} \left[ X_{k_1} X_{l_1} - \Xbar{X_k} X_{l_1} - X_{k_1}  \Xbar{X_l} \right]
\biggr\} \nonumber \\ 
&= \dfrac{1}{4} \biggl\{ 
\Xbar{X_{i} X_{j} X_{k} X_{l}} - \Xbar{X_i} \myspace \Xbar{X_{j} X_{k} X_{l}} -  \Xbar{X_j} \myspace  \Xbar{X_{i}  X_{k} X_{l}} \nonumber \\
&\qquad - \Xbar{X_k} \myspace  \Xbar{X_{i} X_{j}  X_{l}} + \Xbar{X_i} \myspace  \Xbar{X_k} \myspace  \Xbar{ X_{j}  X_{l}} + \Xbar{X_j} \myspace  \Xbar{X_k} \myspace  \Xbar{X_{i}  X_{l}} \nonumber \\
&\qquad - \Xbar{X_{i} X_{j} X_{k}} \myspace \Xbar{X_l} + \Xbar{X_i} \myspace \Xbar{X_l} \myspace \Xbar{ X_{j} X_{k}}   + \Xbar{X_j} \myspace \Xbar{X_l} \myspace \Xbar{X_{i}  X_{k}}    \nonumber \\
&\qquad - \left( \Xbar{X_i X_j} - 2  \myspace \Xbar{X_i} \myspace \Xbar{X_j} \right) \left( \Xbar{X_k X_l} - 2  \myspace\Xbar{X_k} \myspace \Xbar{X_l} \right)
\biggr\} \nonumber
\end{align}

\subsubsection*{Case 2: $i=j$; $j \neq k,l$; $k = l$}

The kernels are 

\noindent\begin{minipage}{.5\linewidth}
\begin{align*}
&h(u_1,u_2) = \dfrac{1}{2} \left( X_{i_1} - X_{i_2} \right)^2;  \\
&\E_{u_2}[h(u_1,u_2)] = \dfrac{1}{2} \left( X_{i_1} - \Xbar{X_i} \right)^2;
\end{align*}
\end{minipage}\hfill
\begin{minipage}{.5\linewidth}
\begin{align*}
&g(v_1,v_2) = \dfrac{1}{2} \left( X_{k_1} - X_{k_2} \right)^2 \\
&\E_{u_2}[g(v_1,v_2)] = \dfrac{1}{2} \left( X_{k_1} - \Xbar{X_k} \right)^2
\end{align*}
\end{minipage}

Then, we have 
\begin{align}
\zeta_1 &= \Cov \left[ 
\dfrac{1}{2} \left( X_{i_1} - \Xbar{X_i} \right)^2;
\dfrac{1}{2} \left( X_{k_1} - \Xbar{X_k} \right)^2
\right] \\
&= \dfrac{1}{4} \biggl\{ \Cov \left[ 
X_{i_1}^2 - 2 X_{i_1} \Xbar{X_i} ;
X_{k_1}^2 - 2 X_{k_1} \Xbar{X_k}
\right] \biggr\} \nonumber \\
&= \dfrac{1}{4} \biggl\{ 
\E_{u_1} \left[ X_{i_1}^2 X_{k_1}^2 - 2 X_{i_1} \Xbar{X_i} X_{k_1}^2 - 2 X_{i_1}^2 X_{k_1} \Xbar{X_k} + 4 X_{i_1} \Xbar{X_i} X_{k_1} \Xbar{X_k} \right] \nonumber \\
&\qquad - \E_{u_1} \left[ X_{i_1}^2 - 2 X_{i_1} \Xbar{X_i}\right] \E_{u_1} \left[ X_{k_1}^2 - 2 X_{k_1} \Xbar{X_k} \right]
\biggr\} \nonumber \\
&= \dfrac{1}{4} \biggl\{ 
\Xbar{ X_{i}^2 X_{k}^2} - 2 \myspace \Xbar{X_i} \myspace \Xbar{X_{i} X_{k}^2}  - 2 \myspace \Xbar{X_{i}^2 X_{k_1}}  \myspace \Xbar{X_k} + 4 \Xbar{X_{i} X_{k}} \myspace \Xbar{X_i}  \myspace \Xbar{X_k} \nonumber \\
&\qquad -  \left( \Xbar{X_{i}^2} - 2 \myspace \Xbar{X_i}^2 \right) \left( \Xbar{X_{k}^2} - 2 \myspace \Xbar{X_k}^2 \right) 
\biggr\} \nonumber
\end{align}

\subsubsection*{Case 3: $i=j$; $j \neq k,l$; $k \neq l$}

The kernels are 

\noindent\begin{minipage}{.5\linewidth}
\begin{align*}
&h(u_1,u_2) = \dfrac{1}{2} \left( X_{i_1} - X_{i_2} \right)^2 ;  \\
&\E_{u_2}[h(u_1,u_2)] = \dfrac{1}{2} \left( X_{i_1} - c \right)^2 ;
\end{align*}
\end{minipage}\hfill
\begin{minipage}{.5\linewidth}
\begin{align*}
&g(v_1,v_2) = \dfrac{1}{2} \left( X_{k_1} - X_{k_2} \right)  \left( X_{l_1} - X_{l_2} \right)  \\
&\E_{u_2}[g(v_1,v_2)] = \dfrac{1}{2} \left( X_{k_1} - \Xbar{X_k} \right)  \left( X_{l_1} - \Xbar{X_l} \right)
\end{align*}
\end{minipage}

Then, we have 
\begin{align}
\zeta_1 &= \Cov\left[ 
\dfrac{1}{2} \left( X_{i_1} - \Xbar{X_i} \right)^2 ; 
 \dfrac{1}{2} \left( X_{k_1} - \Xbar{X_k} \right)  \left( X_{l_1} - \Xbar{X_l} \right)
\right] \\
&= \dfrac{1}{4} \biggl\{ \Cov \left[
X_{i_1}^2 - 2 X_{i_1} \Xbar{X_i} ;
X_{k_1} X_{l_1} - \Xbar{X_k} X_{l_1} - X_{k_1} \Xbar{X_l}
\right] \biggr\} \nonumber \\
&= \dfrac{1}{4} \biggl\{ 
\E_{u_1} \big[  X_{i_1}^2 X_{k_1} X_{l_1} - 2 X_{i_1} \Xbar{X_i} X_{k_1} X_{l_1} - X_{i_1}^2 \Xbar{X_k} X_{l_1} \nonumber \\
&\qquad\qquad + 2 X_{i_1} \Xbar{X_i} \myspace \Xbar{X_k} X_{l_1} - X_{i_1}^2 X_{k_1} \Xbar{X_l} + 2 X_{i_1} \Xbar{X_i} X_{k_1} \Xbar{X_l} \big] \nonumber \\
&\qquad - \E_{u_1} \left[ X_{i_1}^2 - 2 X_{i_1} \Xbar{X_i}  \right] 
\E_{u_1} \left[ X_{k_1} X_{l_1} - \Xbar{X_k} X_{l_1} - X_{k_1} \Xbar{X_l} \right] 
\biggr\} \nonumber \\
&= \dfrac{1}{4} \biggl\{ 
\Xbar{ X_{i}^2 X_{k} X_{l} } - 2 \myspace \Xbar{X_{i} X_{k} X_{l} } \myspace \Xbar{X_i} - \Xbar{X_{i}^2 X_{l}} \myspace \Xbar{X_k}  \nonumber \\
&\qquad\qquad + 2 \myspace \Xbar{X_{i} X_{l}} \myspace \Xbar{X_i} \myspace \Xbar{X_k}  - \Xbar{X_{i}^2 X_{k_1}} \myspace \Xbar{X_l} + 2 \myspace \Xbar{X_{i} X_{k}} \myspace \Xbar{X_i} \myspace \Xbar{X_l} \nonumber \\
&\qquad - \left( \Xbar{X_{i}^2} - 2 \myspace \Xbar{X_i}^2 \right) \left( \Xbar{X_{k} X_{l} } - 2 \myspace \Xbar{X_k} \myspace \Xbar{X_l}\right)
\biggr\} \nonumber
\end{align}

\subsubsection*{Case 4: $i=k$; $j \neq i,k,l$; $k \neq l$}

The kernels are 

\noindent\begin{minipage}{.5\linewidth}
\begin{align*}
&h(u_1,u_2) = \dfrac{1}{2} \left( X_{i_1} - X_{i_2} \right) \left( X_{j_1} - X_{j_2} \right)  ;  \\
&\E_{u_2}[h(u_1,u_2)] = \dfrac{1}{2} \left( X_{i_1} - \Xbar{X_i} \right) \left( X_{j_1} - \Xbar{X_j}\right) ;
\end{align*}
\end{minipage}\hfill
\begin{minipage}{.5\linewidth}
\begin{align*}
&g(v_1,v_2) = \dfrac{1}{2} \left( X_{i_1} - X_{i_2} \right)  \left( X_{l_1} - X_{l_2} \right) \\
&\E_{u_2}[g(v_1,v_2)] = \dfrac{1}{2} \left( X_{i_1} - \Xbar{X_i} \right)  \left( X_{l_1} - \Xbar{X_l} \right)
\end{align*}
\end{minipage}

Then, we have 
\begin{align}
\zeta_1 &= \Cov \left[ 
\dfrac{1}{2} \left( X_{i_1} - \Xbar{X_i} \right) \left( X_{j_1} - \Xbar{X_j}\right)  ;
\dfrac{1}{2} \left( X_{i_1} - \Xbar{X_i} \right)  \left( X_{l_1} - \Xbar{X_l} \right)
\right] \\
&= \dfrac{1}{4} \biggl\{ \Cov \left[ 
X_{i_1} X_{j_1} - \Xbar{X_i} X_{j_1} - X_{i_1} \Xbar{X_j}  ;
X_{i_1} X_{l_1} - \Xbar{X_i} X_{l_1} - X_{i_1} \Xbar{X_l} 
\right] \biggr\} \nonumber \\
&= \dfrac{1}{4} \biggl\{ \E_{u_1} \big[ 
X_{i_1}^2 X_{j_1} X_{l_1} - \Xbar{X_i} X_{j_1} X_{i_1} X_{l_1} - X_{i_1}^2 \Xbar{X_j} X_{l_1} \nonumber \\ 
&\qquad\qquad - X_{i_1} X_{j_1} \Xbar{X_i} X_{l_1} + \Xbar{X_i}^2 \myspace X_{j_1} X_{l_1} + X_{i_1} \Xbar{X_j} \myspace \Xbar{X_i} X_{l_1} \nonumber \\
&\qquad\qquad - X_{i_1}^2 X_{j_1} \Xbar{X_l}  + \Xbar{X_i} X_{j_1} X_{i_1} \Xbar{X_l}  + X_{i_1}^2 \Xbar{X_j}  \Xbar{X_l} 
\big] \nonumber \\
&\qquad -  \E_{u_1} \left[ X_{i_1} X_{j_1} - \Xbar{X_i} X_{j_1} - X_{i_1} \Xbar{X_j} \right]  
\E_{u_1} \left[ X_{i_1} X_{l_1} - \Xbar{X_i} X_{l_1} - X_{i_1} \Xbar{X_l} \right] 
\biggr\} \nonumber \\
&= \dfrac{1}{4} \biggl\{ 
\Xbar{X_{i_1}^2 X_{j_1} X_{l_1}} - \Xbar{X_i} \myspace \Xbar{X_{j_1} X_{i_1} X_{l_1}} - \Xbar{X_{i_1}^2 X_{l_1} } \myspace \Xbar{X_j} \nonumber \\ 
&\qquad\qquad - \Xbar{X_{i_1} X_{j_1} X_{l_1}} \myspace \Xbar{X_i}  + \Xbar{X_i}^2 \myspace \Xbar{X_{j_1} X_{l_1}}  + \Xbar{X_{i_1} X_{l_1}} \myspace \Xbar{X_j} \myspace \Xbar{X_i}  \nonumber \\
&\qquad\qquad - \Xbar{X_{i_1}^2 X_{j_1} } \myspace \Xbar{X_l}  + \Xbar{X_i} \myspace \Xbar{X_{j_1} X_{i_1}} \myspace \Xbar{X_l}  + \Xbar{X_{i_1}^2} \myspace \Xbar{X_j} \myspace  \Xbar{X_l} 
\big] \nonumber \\
&\qquad -  \left( \Xbar{X_{i} X_{j}} - 2 \myspace \Xbar{X_i} \myspace \Xbar{X_j} \right)
\left(  \Xbar{X_{i} X_{l}} - 2 \myspace \Xbar{X_i} \myspace \Xbar{X_l} \right) 
\biggr\} \nonumber
\end{align}

\subsubsection*{Case 5: $i=k$; $i \neq j$; $j=l$; }

\noindent\begin{minipage}{.5\linewidth}
\begin{align*}
&h(u_1,u_2) = \dfrac{1}{2} \left( X_{i_1} - X_{i_2} \right) \left( X_{j_1} - X_{j_2} \right);  \\
&\E_{u_2}[h(u_1,u_2)] = \dfrac{1}{2} \left( X_{i_1} - \Xbar{X_i} \right) \left( X_{j_1} - \Xbar{X_j} \right);
\end{align*}
\end{minipage}\hfill
\begin{minipage}{.5\linewidth}
\begin{align*}
&g(v_1,v_2) = h(u_1,u_2) \\
&\E_{u_2}[g(v_1,v_2)] = \E_{u_2}[h(u_1,u_2)]
\end{align*}
\end{minipage}

Then, we have 
\begin{align}
\zeta_1 &= \Var \left[ \dfrac{1}{2} \left( X_{i_1} - \Xbar{X_i} \right) \left( X_{j_1} - \Xbar{X_j} \right) \right]  \\
&= \dfrac{1}{4} \biggl\{ \Var \left[ X_{i_1} X_{j_1} - \Xbar{X_i} X_{j_1} - X_{i_1} \Xbar{X_j}  \right]  \biggr\} \nonumber \\
&= \dfrac{1}{4} \biggl\{ \E_{u_1} \left[ (X_{i_1} X_{j_1} - \Xbar{X_i} X_{j_1} - X_{i_1}  \Xbar{X_j} )^2 \right]- \E_{u_1} \left[ X_{i_1} X_{j_1} - \Xbar{X_i} X_{j_1} - X_{i_1} \Xbar{X_j}  \right]^2   \biggr\}  \nonumber \\
&= \dfrac{1}{4} \biggl\{ \E_{u_1} \big[ X_{i_1}^2 X_{j_1}^2 - 2 X_{i_1} X_{j_1}^2 \Xbar{X_i} + \Xbar{X_i}^2 X_{j_1}^2 
 - 2 X_{i_1}^2 X_{j_1}  \Xbar{X_j} + 2 \Xbar{X_i} X_{j_1} X_{i_1}  \Xbar{X_j} + X_{i_1}^2  \Xbar{X_j}^2 \big]  \nonumber \\
&\qquad \qquad  - \left( \Xbar{X_i X_j} - 2 (\Xbar{X_i} \myspace \Xbar{X_j}) \right)^2   \biggr\}  \nonumber \\
&= \dfrac{1}{4} \biggl\{  \Xbar{X_{i}^2 X_{j}^2} - 2 \Xbar{X_{i} X_{j}^2} \myspace \Xbar{X_i} + \Xbar{X_i}^2 \myspace\Xbar{X_{j}^2} 
 - 2 \Xbar{X_{i}^2 X_{j}} \myspace  \Xbar{X_j} + 2 \Xbar{X_i} \myspace \Xbar{X_j}  \myspace\Xbar{X_{j} X_{i}} + \Xbar{X_{i}^2} \myspace  \Xbar{X_j}^2 \nonumber \\
&\qquad \qquad - \left( \Xbar{X_i X_j} - 2 (\Xbar{X_i} \myspace \Xbar{X_j}) \right)^2   \biggr\}  \nonumber
\end{align}

\subsubsection*{Case 6: $i=j=k$; $i \neq l$}

The kernels are 

\noindent\begin{minipage}{.5\linewidth}
\begin{align*}
&h(u_1,u_2) = \dfrac{1}{2} \left( X_{i_1} - X_{i_2} \right)^2 ;  \\
&\E_{u_2}[h(u_1,u_2)] = \dfrac{1}{2} \left( X_{i_1} - \Xbar{X_i} \right)^2 ;
\end{align*}
\end{minipage}\hfill
\begin{minipage}{.5\linewidth}
\begin{align*}
&g(v_1,v_2) = \dfrac{1}{2} \left( X_{i_1} - X_{i_2} \right)  \left( X_{l_1} - X_{l_2} \right) \\
&\E_{u_2}[g(v_1,v_2)] = \dfrac{1}{2} \left( X_{i_1} -\Xbar{X_i} \right)  \left( X_{l_1} - \Xbar{X_l} \right)
\end{align*}
\end{minipage}

Then, we have 
\begin{align}
\zeta_1 &= \Cov \left[
\dfrac{1}{2} \left( X_{i_1} - \Xbar{X_i} \right)^2 ;
\dfrac{1}{2} \left( X_{i_1} -\Xbar{X_i} \right)  \left( X_{l_1} - \Xbar{X_l} \right)
\right]  \\
&= \dfrac{1}{4} \biggl\{  \Cov \left[
X_{i_1}^2 - 2 X_{i_1} \Xbar{X_i};
X_{i_1} X_{l_1} - \Xbar{X_i} X_{l_1} - X_{i_1} \Xbar{X_l}
\right] \biggr\}  \nonumber \\
&= \dfrac{1}{4} \biggl\{  \E_{u_1} \big[
X_{i_1}^2 X_{i_1} X_{l_1} - 2 X_{i_1} \Xbar{X_i} X_{i_1} X_{l_1} - X_{i_1}^2 \Xbar{X_i} X_{l_1} \nonumber \\
&\qquad \qquad + 2 X_{i_1} \Xbar{X_i} \myspace \Xbar{X_i} X_{l_1} - X_{i_1}^2 X_{i_1} \Xbar{X_l} + 2 X_{i_1} \Xbar{X_i} X_{i_1} \Xbar{X_l} \big] \nonumber \\
&\qquad - \E_{u_1} \left[ X_{i_1}^2 - 2X_{i_1} \Xbar{X_i} \right] 
\E_{u_1} \left[ X_{i_1} X_{l_1} - \Xbar{X_i} X_{l_1} - X_{i_1} \Xbar{X_l} \right]  
\biggr\}  \nonumber \\
&= \dfrac{1}{4} \biggl\{  
\Xbar{X_{i}^3 X_{l}} - 3 \myspace \Xbar{X_{i}^2 X_{l}} \myspace \Xbar{X_i}  + 2 \myspace \Xbar{ X_{i} X_{l}} \myspace \Xbar{X_i}^2 - \Xbar{X_{i}^3} \myspace \Xbar{X_l} + 2 \myspace \Xbar{X_{i}^2} \myspace \Xbar{X_i} \myspace \Xbar{X_l}  \nonumber \\
&\qquad - \left( \Xbar{X_i^2} - 2 \myspace  \Xbar{X_i}^2 \right)
\left( \Xbar{X_i X_l} - 2 \myspace  \Xbar{X_i} \myspace \Xbar{X_l} \right)
\biggr\}  \nonumber
\end{align}

\subsubsection*{Case 7: $i=j,k,l$}

The kernels are 

\noindent\begin{minipage}{.5\linewidth}
\begin{align*}
&h(u_1,u_2) = \dfrac{1}{2} \left( X_{i_1} - X_{i_2} \right)^2 ;  \\
&\E_{u_2}[h(u_1,u_2)] = \dfrac{1}{2} \left( X_{i_1} - \Xbar{X_i} \right)^2 ;
\end{align*}
\end{minipage}\hfill
\begin{minipage}{.5\linewidth}
\begin{align*}
&g(v_1,v_2) = h(u_1,u_2) \\
&\E_{u_2}[g(v_1,v_2)] =\E_{u_2}[h(u_1,u_2)]
\end{align*}
\end{minipage}

Then, we have 
\begin{align}
\zeta_1 &= \Var \left[ \dfrac{1}{2} \left( X_{i_1} - \Xbar{X_i} \right)^2 \right]  \\
&= \dfrac{1}{4} \Var \left[  X_{i_1}^2 - 2 X_{i_1} \Xbar{X_i} \right] \nonumber \\
&= \dfrac{1}{4} \biggl\{
\E_{u_1} \left[ \left( X_{i_1}^2 - 2 X_{i_1} \Xbar{X_i} \right)^2\right]
- \E_{u_1} \left[ X_{i_1}^2 - 2 X_{i_1} \Xbar{X_i} \right]^2
\biggr\}\nonumber \\
&= \dfrac{1}{4} \biggl\{
\Xbar{X_{i}^4} - 4 \Xbar{X_{i}^3} \myspace \Xbar{X_i} + 4 \Xbar{X_{i}^2} \myspace \Xbar{X_i}^2 
- \left( \Xbar{X_i^2} - 2 \Xbar{X_i}^2 \right)^2
\biggr\}\nonumber
\end{align}


\subsection{Derivation in $\mathcal{O}(n)$ time for all terms} \label{aistats2021:appendix:subsec:proof_linear_computation_time}

In section~\ref{aistats2021:appendix:subsec:derivation_sevencases}, all terms are in the form of $\E[X]$,$\E[XY]$,$\E[XYZ]$ and $\E[XYUV]$ and can be computed in $\mathcal{O}(n)$ as following

\begin{align}
\E[X] &\approx \frac{1}{m} \sum_{q=1}^n X_q \\
\E[XY] &\approx \frac{1}{m} \sum_{q=1}^n X_q \odot Y_q \\
\E[XYZ] &\approx \frac{1}{m} \sum_{q=1}^n X_q \odot Y_q \odot Z_q \\
\E[XYUV] &\approx \frac{1}{m} \sum_{q=1}^n X_q \odot Y_q \odot U_q \odot V_q
\end{align}

\section{Orthonormality constraints}
\label{section:orthonormality}

The steps for implementing the orthonormality constraints are summarized in Algorithm~\ref{algo:orthonormality}.

\begin{algorithm}[ht!]
    \footnotesize
    \begin{algorithmic}[1]
        \Require{Previously obtained lower and upper bound on $V \odot V$, $\hat{\beta}$ and $\hat{\alpha}$ respectively; $N$ - number of eigenvectors} 
        
      \State Update $\hat{\beta}$ and $\hat{\alpha}$ using Equation~\eqref{eq:TightenedBoundsFromNormConstraints}
      \State Obtain $\beta$ and $\alpha$, the signed lower and upper bound respectively, using Corollary~\ref{cor:signedBounds}
       
        \State $\mu, \nu \leftarrow \textsc{GetBoundsOnSum}(\alpha, \beta) $
        \For {$l=1,2,\ldots, N$} 
            \For {$i=1,2,\ldots,N$}
                \For {$j=1,2, \ldots, N, j \neq i $}
                \If {$\alpha_{lj} < 0$}
                    \State $\beta_{li} = \max \left (\beta_{li},  
                    \min(-\frac{\mu_{ijl}}{\alpha_{lj}}, -\frac{\mu_{ijl}}{\beta_{lj}} ) \right ) $ 
                    
                     \State $\alpha_{li} = \min \left (\alpha_{li},  \min(-\frac{\nu_{ijl}}{\alpha_{lj} }, -\frac{\nu_{ijl}}{\beta_{lj}} ) \right ) $ 
                \EndIf
                
                \If {$\beta_{lj} > 0 $}
                  \State $\beta_{li} = \max \left (\beta_{li},  
                    \min(-\frac{\nu_{ijl}}{\alpha_{lj}}, -\frac{\nu_{ijl}}{\beta_{lj}} ) \right ) $ 
                    
                    \State $\alpha_{li} = \min \left (\alpha_{li},  \min(-\frac{\mu_{ijl}}{\alpha_{lj} }, -\frac{\mu_{ijl}}{\beta_{lj}} ) \right ) $ 
                \EndIf
                
                \If {$\alpha_{lj} > 0 \wedge \beta_{lj} < 0$}
                    \If {$\beta_{li} \geq 0$}
                        $\beta_{li} = \max \left (\beta_{li},  
                        -\frac{\mu_{ijl}}{\beta_{lj}}, -\frac{\nu_{ijl}}{\alpha_{lj}} \right ) $ 
                    \EndIf
                    \If {$\alpha_{li} < 0$}
                        $\alpha_{li} = \min \left (\alpha_{li},  -\frac{\mu_{ijl}}{\alpha_{lj} }, -\frac{\nu_{ijl}}{\beta_{lj}} \right ) $ 
                    \EndIf               
                \EndIf
                
                \If {$\beta_{lj} = 0$}
                    \If {$\nu_{ijl} < 0$} $\beta_{li} = \max \left (\beta_{li}, -\frac{\nu_{ijl}}{\alpha_{lj} } \right ) $ \EndIf
                    \If {$\mu_{ijl} > 0$} $\alpha_{li} = \min \left (\alpha_{li}, 
                        -\frac{\mu_{ijl}}{\alpha_{lj} } \right ) $  \EndIf

                \EndIf
                
                \If {$\alpha_{lj} = 0$}
                    \If {$\mu_{ijl} > 0$} $\beta_{li} = \max \left (\beta_{li}, 
                        -\frac{\mu_{ijl}}{\beta_{lj} } \right ) $  \EndIf
                    \If {$\nu_{ijl} < 0$} $\alpha_{li} = \min \left (\alpha_{li}, 
                        -\frac{\nu_{ijl}}{\beta_{lj} } \right ) $ \EndIf
                \EndIf
            	
            	\EndFor
        	\EndFor
    	\EndFor
      \\\hrulefill
      \Function {GetBoundsOnSum}{$\alpha$, $\beta$}
        	\For {$l=1,2,\ldots, N$} 
            	\For {$i=1,2,\ldots,N-1$}
                	\For {$j=i+1 \ldots, N$}
                	    \State $\mu_{ijl} = \sum_{k\neq l} \min(\beta_{ki} \beta_{kj}, \beta_{ki} \alpha_{kj}, \alpha_{ki} \beta_{kj}, \alpha_{ki} \alpha_{kj})$ 
                	    \State $\nu_{ijl} = \sum_{k\neq l} \max(\beta_{ki} \beta_{kj}, \beta_{ki} \alpha_{kj}, \alpha_{ki} \beta_{kj}, \alpha_{ki} \alpha_{kj})$ 
                	\EndFor
            	\EndFor
        	\EndFor
        	\State \Return $\mu$, $\nu$
    	\EndFunction

    \end{algorithmic}
    
    \caption{Orthonormality constraints.}
    \label{algo:orthonormality}
\end{algorithm}

\section{Synthetic Data Simulation Details}
\begin{enumerate}
    \item Generate a random adjacency matrix $A$ with arbitrary sparsity, such that $A_{ij}\sim \textrm{Bernoulli}(t), t\sim \mathcal{U}(0, 1)$.
    \item Generate a random affinity matrix $\hat A = A * c * (\mathbf{1} + R)$, where $c$ is some positive constant, $\mathbf{1}$ -- matrix of all ones, $R_{ij}\sim\mathcal{U}(0, 1)$. 
    \item Do the eigendecomposition, such that $\hat A=VSV^\intercal$, set $\hat S_{ii}=\max(S_{ii}, 1 + |\nu|), \nu \sim \mathcal{U}(0, 1)$.
    \item $\Theta=V\hat SV^\intercal$. If $\Theta$ does not have at least 1 zero, the sample is rejected.
\end{enumerate}

We generate the data $Y$ from the Laplace distribution as follows~\citep{kotz2012laplace}:
\begin{enumerate}
    \item Generate $p$-dimensional standard exponential variate $W\sim -\log \left[\mathcal{U}(0, 1)\right]$.
    \item Generate variable $Z\sim\mathcal{N}(0, \Theta^{-1})$.
    \item $Y = m \cdot W + \sqrt{W}\cdot Z$, where $m$ defines the location of the distribution.
\end{enumerate}

\bibliography{refs}


\end{document}